\documentclass[english,10pt]{amsart}
\usepackage[francais]{babel}
\usepackage{amsmath,amsfonts,amssymb,multicol, amsthm}
\usepackage{amsrefs}
\usepackage{mathrsfs}
\usepackage[utf8]{inputenc}
\usepackage{esint}
\usepackage{color, framed}

\usepackage{graphicx}
\usepackage[normalem]{ulem}

\newcommand{\R}{\ensuremath{\mathbb{R}}}

\newcommand{\N}{\ensuremath{\mathbb{N}}}

\renewcommand{\leq}{\leqslant}
\renewcommand{\geq}{\geqslant}

\newcommand{\eps}{\epsilon}

\newtheorem{theorem}{Theorem}
\newtheorem{definition}{Definition}
\newtheorem{thmx}{Theorem}

\newtheorem{lemma}[theorem]{Lemma}
\newtheorem{proposition}[theorem]{Proposition}

\theoremstyle{remark}
\newtheorem{remark}[theorem]{Remark}

\newcommand{\super}{\overline}

\newcounter{numeroexo}

\title{Regularity and symmetry  for semilinear elliptic equations in bounded domains.}
\author{Louis Dupaigne}
\address{Institut Camille Jordan, UMR CNRS 5208, Universit\'e Claude Bernard Lyon 1, 43 boulevard du 11 novembre 1918, 69622 Villeurbanne cedex, France}
\email{dupaigne@math.univ-lyon1.fr}
\author{Alberto Farina}
\address{LAMFA, UMR CNRS 7352, Universit\'e Picardie Jules Verne 33, rue St Leu, 80039 Amiens, France}
\email{alberto.farina@u-picardie.fr}
\date{}

\usepackage[T1]{fontenc}
\begin{document}

\renewcommand{\appendixname}{Appendix}
\renewcommand\refname{References}

\maketitle

\begin{abstract}
In the present paper, we investigate the regularity and symmetry properties of weak solutions to semilinear elliptic equations which are locally stable. 
\end{abstract}

\section{Introduction and main results} 
In the present paper, we investigate the regularity and symmetry properties of weak solutions to semilinear elliptic equations.  We shall focus on the following class:
\begin{definition}\label{loc stab} Let  $N\ge1$, $\Omega\subset\R^N$ denote an open set and $f\in C^1(\R)$.  Assume that $u\in H^1_{loc}(\Omega)$, $f(u)\in L^1_{loc}(\Omega)$ and that $u$ solves 
\begin{equation}\label{equazione}
-\Delta u = f(u) \qquad\text{in $\mathcal D'(\Omega)$}.
\end{equation}
We say that $u$ is \textbf{locally stable} in $\Omega$ if $f'(u)\in L^1_{loc}(\Omega)$ and if for every $x\in \Omega$, there exists an open neighborhood $\omega\subset\Omega$ of $x$ such that for every $\varphi\in C^1_c(\omega)$, there holds
\begin{equation}\label{stab}
\int_{\omega}f'(u)\varphi^2 \le \int_{\omega}\vert\nabla\varphi\vert^2.
\end{equation}
A solution is stable in $\Omega$ if the above inequality holds for $\omega=\Omega$ and for every $\varphi\in C^1_c(\Omega)$.
\end{definition}


\medskip

As shown by the following examples, the class of locally stable solutions is natural and wide enough to encompass various interesting families of solutions (naturally) arising in the study of PDEs. 

\medskip

\begin{enumerate}
	\item
	Smooth solutions are locally stable, thanks to the (sharp) Poincaré inequality. 
	\item More generally, for $N\ge3$, weak solutions such that $f'(u) \in L^1_{loc}(\Omega)$ and $f'(u)^+\in L^{N/2}_{loc}(\Omega)$ are locally stable. 
	Indeed, choosing $\omega$ so small that $ \Vert f'(u)^+\Vert_{L^{\frac N2}(\omega)}\le \frac{N(N-2)\vert B_1\vert}4$ and applying Hölder's and Sobolev's inequalities, we have
	$$
	\int_{\omega}f'(u)\varphi^2\le \int_{\omega}f'(u)^+\varphi^2\le \Vert f'(u)^+\Vert_{L^{\frac N2}(\omega)} \Vert \varphi\Vert^2_{L^{\frac{2N}{N-2}}(\omega)}\le \int_{\omega}\vert\nabla\varphi\vert^2
	$$
	When $N=2,$ the local stability follows from Moser-Trudinger inequality if $f'(u) \in L^1_{loc}(\Omega)$ and $f'(u)^+\in L^{p}_{loc}(\Omega)$ for some $p>1$.
	\item If $N \geq 3,$ $f(u)=2(N-2)e^u$ and $u=-2\ln\vert x\vert$, then $f'(u)=\frac{2(N-2)}{\vert x\vert^2} \in L^1_{loc}$ but just fails to belong to $L^{N/2}$ near the origin. By the optimality of Hardy's inequality, $u$ is never locally stable in any open set containing the origin whenever $3\le N\le9$.
	\item Local minimizers are stable: $u\in H^1_{loc}(\Omega)$ is a local minimizer if for any $ \Omega' \subset \subset \Omega$ and for all $\varphi\in C^1_c(\Omega')$, $t=0$ is a point of minimum of the function $t\mapsto e(t):=\mathcal E_{\Omega'}(u+t\varphi)$, where
	$\mathcal E_{\Omega'}(v)=\int_{\Omega'}\left(\frac12\vert\nabla v\vert^2-F(v)\right)$ and $F'=f$. Therefore \eqref{stab} holds (since $e''(0)\ge0)$. 
	\item If $u\in H^1_{loc}(\Omega)$ has finite Morse index\footnote{We recall that a solution $u$ to \eqref{equazione} has Morse index equal to $K \ge 0$, if $f'(u)\in L^1_{loc}(\Omega)$ and $K$ is the maximal dimension of a subspace $X_K$ of $ C^1_c(\Omega) $ such that $ \int_{\Omega}\vert \nabla \psi \vert^2 < \int_{\Omega} f'(u) \psi^2$ for any $\psi \in X_K \setminus \{0\}$. In particular $u$ is stable if and only if its Morse index is zero. 
	} in $\Omega$, then $u$ is locally stable in $\Omega$, see Proposition 1.5.1 in \cite{dup} (or Proposition 2.1 in \cite{ddf}). { In addition, $u$ is stable outside a compact set, see Remark 1 in \cite{farina}.} But there are also locally stable solutions of infinite Morse index. This is the case e.g. when $\Omega$ is the punctured unit ball, $f(u)=2(N-2)e^u$, $u(x)=-2\ln\vert x\vert$ and $3\le N\le 9$.
\end{enumerate}

\medskip

Our first result concerns the complete classification of nonnegative stable solutions $u \in H^1_0(\Omega)$ to \eqref{equazione}, when $f$ is a convex function satisfying $f(0)=0$. 

\medskip

\begin{theorem}\label{th:stable:f(0)=0}
	Let $\Omega$ be a bounded domain of $\R^N$, $N \geq 1,$ let $f \in C^1(\R)$ be a convex function such that $f(0)=0$ and let $\lambda_1$ be the principal eigenvalue of $ -\Delta$ with homogeneous Dirichlet boundary conditions.  Assume that $u\in H^1_0(\Omega)$, $f(u) \in L^1_{loc}(\Omega)$ and that $u$ is a stable solution to
	\begin{equation}\label{eq:Lemma0}
	\left\{
	\begin{aligned}
	-\Delta u& = f(u) \qquad\text{in $\mathcal D'(\Omega)$}\\
	u&\geq 0 \qquad\text{a.e. on $\Omega$}.
	\end{aligned}
	\right.
	\end{equation} 
	Then, either $u \equiv 0$ or $f(t) = \lambda_1 t $ on $(0, \sup_\Omega u)$ and $u \in C^\infty(\Omega) \cap H^1_0(\Omega)$ is a positive first eigenfunction of $ -\Delta$ with homogeneous Dirichlet boundary conditions.   
\end{theorem}

\begin{remark}
	If $u \equiv0$, then necessarily $ f'(0) \leq \lambda_1,$ by Lemma \ref{lemma lambda1} in section \ref{S3} below. Also observe that for any $\alpha \leq \lambda_1$ there is a convex function $f$ satisfying $f(0)=0$, $f'(0)=\alpha $ and such that $u \equiv 0$ is a stable solution to \eqref{eq:Lemma0}. An example is provided by $f(u) = u^2 +\alpha u$. 
\end{remark}

\medskip

The latter result is a consequence of the following general theorem which holds true for any convex function $f$ of class $C^1$ and for distributional solutions merely in $H^1(\Omega)$.

\medskip

\begin{theorem} \label{theorem:compa-stable-solutions} Let $\Omega$ be a bounded domain of $\R^N$, $N \geq 1$ and let $f \in C^1([0;+\infty))$ be a convex function. Assume that $u,v \in H^1(\Omega)$ satisfy $ u-v \in H^1_0(\Omega),$ $0 \leq v \leq u $ a.e. on $\Omega,$ $f(u),f(v) \in L^1_{loc}(\Omega)$ and both $u$ and $v$ are solution to 
	\begin{equation}\label{eq-D1}
	-\Delta w = f(w)\qquad\text{in $\mathcal D'(\Omega)$}.
	\end{equation}
	If $f'(u) \in L^1_{loc}(\Omega)$ and $u$ is stable, then either $u\equiv v$ or $f(t) = a + \lambda_1 t$ for all $t \in (\inf_{\Omega} v, \, \sup_{\Omega} u)$ and some\footnote{Actually the real number $a$ is unique and its value is given by $- \lambda_1\frac{\int_\Omega \phi_1 h}{\int_\Omega \phi_1} \leq 0, $ where $\phi_1$ is a positive first eigenfunction of $ -\Delta$ with homogeneous Dirichlet boundary conditions and $h \in H^1(\Omega)$ is the unique weak solution of $ - \Delta h =0$ in $ \Omega$ with $ u-h \in H^1_0(\Omega)$ (to see this, use $\phi_1$ as test function in the weak formulation of $-\Delta u = a + \lambda_1 u$ and the fact that $h$ is harmonic) and also note that $h$ is nonnegative by the maximum principle. In particular, $u,v \in H^1_0(\Omega) \Longleftrightarrow a=0$. Also note that, for every $ a \leq 0$ there exist solutions $u,v$ for which the second alternative of the theorem occurs. Indeed, the functions $u_t := - \frac{a}{\lambda_1} + t \phi_1$, $ t \geq 0$ are suitable.}
	$a \in \R,$ $u,v \in C^\infty(\Omega)$ and $u-v$ is a positive first eigenfunction of $ -\Delta$ with homogeneous Dirichlet boundary conditions.   
\end{theorem}

\begin{remark}According to Theorem 1.3 and Corollary 3.7 in \cite{caz}, there exists a $C^1$, positive, increasing but non-convex nonlinearity $f$ with two distinct and ordered (classical) stable\footnote{Indeed, applying Corollary 3.7 in \cite{caz}, we see that in the notations of that corollary, for $f(u)=\underline\lambda\overline g(u)$, $\overline u_\lambda$ is minimal hence stable. In addition, by minimality, $\lambda\mapsto \overline u_\lambda$ is nondecreasing and so $\overline u_\lambda$ converges to a stable solution $v$ as $\lambda\searrow\underline\lambda$ such that $v\ge \overline u_{\underline\lambda}$. Then, take $u= \overline u_{\underline\lambda}$} solutions $0\le u\le v$. In other words, the convexity assumption cannot be completely removed from the above theorem.
\end{remark}

\medskip

Another important consequence of Theorem \ref{theorem:compa-stable-solutions} is the following approximation result which, in turn, motivated our definition of local stability (see Definition \ref{loc stab}). This result will be central in the proof of our main regularity results for locally stable solutions to \eqref{equazione}.  

\medskip

\begin{theorem}\label{Prop-approx-crucial} 
	Assume $ \alpha \in (0,1)$ and $N \geq 1$. 
	
	Let $\Omega$ be a bounded domain of $\R^N$ and let $f\in C^1([0,+\infty))$ be a convex function such that $f(0)\geq 0$. Assume that $u \in H^1(\Omega),$ $f(u) \in L^1_{loc}(\Omega)$ and that $u$ is a stable solution to 
	\begin{equation}\label{equazione1} 
	\left\{
	\begin{aligned}
	-\Delta u& = f(u) \qquad\text{in $\mathcal D'(\Omega)$}\\
	u&\geq 0 \qquad\text{a.e. on $\Omega$}.
	\end{aligned}
	\right.
	\end{equation} 
	Then, there is a nondecreasing sequence $(f_k)$ of convex functions in $C^1([0,+\infty)) \cap C^{0,1}([0,+\infty))$ such that $ f_k \nearrow f$ pointwise in $[0; +\infty)$ and a nondecreasing sequence $(u_k)$ of functions in $H^1(\Omega) \cap C^2(\Omega)$ such that $u_k$ is a stable weak solution\footnote{That is a function $u_k$ satisfying $\int_\Omega \nabla u_k \nabla \varphi = \int_\Omega f_k(u_k) \varphi$, for all $ \varphi \in H^1_0(\Omega)$. } to 
	\begin{equation}\label{equazione-k}
	\left\{
	\begin{aligned}
	-\Delta&u_k = f_k(u_k) \quad \textit{in} \quad \Omega,\\
	&u_k- u \in H^1_0(\Omega),\\
	&0 \leq u_k \leq u \qquad\text{a.e. on $\Omega$},
	\end{aligned}
	\right.
	\end{equation}
	and
	\begin{equation}\label{convergence}
	u_k \longrightarrow u \quad \text{in} \quad H^1(\Omega), \qquad 
	\qquad u_k \nearrow u  \quad {a.e.} \,\, \text{on} \,\, \Omega.
	\end{equation}
	
	\medskip
	
	Moreover, if $f$ is nonnegative, then any function $f_k$ is nonnegative too.
	 
\end{theorem}

\medskip

\begin{remark}
	\begin{enumerate}
		\item It follows from \eqref{convergence} that under the assumptions of the proposition, locally stable solutions are automatically lower semi-continuous.
		\item The proposition recovers and extends Corollary 3.2.1. in \cite{dup}.
		\item The result is {\it not true} if we drop the assumption $u\in H^1(\Omega)$. To see this, consider Example 3.2.1 in \cite{dup} in the light of Theorem \ref{smoothness} below. 
		\item We do not know if the assumption $f$ convex can be dropped. 
		\item Theorem \ref{Prop-approx-crucial} generalizes Proposition \ref{Prop-approx-crucial2} below, in which the approximating nonlinearity is taken of the form $f_k=(1-\epsilon_k)f$, with $\epsilon_k\to0$ at the expense of additionally assuming that $f$ is nondecreasing.
	\end{enumerate}
\end{remark}

\medskip

Theorem \ref{Prop-approx-crucial} can be combined with the following {\it a priori} estimate due to \cite{cfrs} in order to establish smoothness of locally stable solutions when $f$ is nonnegative, convex and $N\le 9$.

\medskip

\begin{thmx}[\cite{cfrs}]\label{th:cfrs} Let $B_1$ be the unit ball of $\mathbb{R}^N$, $ N \geq 1$. Assume that $u\in C^2(B_1)$ is a stable solution of \eqref{equazione} in $\Omega=B_1$,
	where $f:\R\to\R$ is locally Lipschitz and nonnegative. If $1\le N\le 9$, then 
	\begin{equation}\label{cfrs}
	\Vert u\Vert_{C^\alpha(\super{B_{1/2}})} \le C \Vert u\Vert_{L^1(B_{1})},
	\end{equation}
	where $\alpha\in(0,1)$, $C>0$ are dimensional constants.
\end{thmx}


\medskip

More precisely we  have the following interior regularity result:

\medskip

\begin{theorem}\label{smoothness} Let $f\in C^1([0,+\infty))$ be a nonnegative convex function.
Let $\Omega$ be an open set of $\R^N, N \geq 1$.
Assume that $u\in H^1_{loc}(\Omega)$, $f(u)\in L^1_{loc}(\Omega)$ and that $u$ is a locally stable solution of \eqref{equazione} such that $u\ge 0$ a.e. in $\Omega$. \\ If $1 \leq N \leq 9$, then  $u\in C^{2,\beta}_{loc}(\Omega)$ for all $\beta\in(0,1)$. 
In particular, any finite Morse index solution is smooth in $\Omega$. 

\end{theorem}

\begin{remark}
	The result is optimal since for $N\ge10$, $f(u)=2(N-2)e^u$ and $\Omega=B_1$, $u(x)=-2\ln\vert x\vert$ is a singular stable solution in $H^1_0(B_1)$. Also observe that the above theorem fails if we do not assume that $u$ belongs to $H^1_{loc}(\Omega)$, see e.g. Example 3.2.1 in \cite{dup}. 
\end{remark}

A priori estimates near the boundary are more subtle, as the following result shows.

\begin{theorem}\label{smoothness boundary} Let $f\in C^1([0,+\infty))$ be a nonnegative convex function. Let $\Omega$ be an open set of $\R^N, N \geq 1$.
Assume that $u\in H^1_{0}(\Omega)$, $f(u)\in L^1_{loc}(\Omega)$ and that $u$ is a finite Morse index  solution of \eqref{equazione} such that $u\ge 0$ a.e. in $\Omega$.
	\begin{enumerate}
\item  Let $\Omega$ be a bounded  uniformly convex domain of class $C^{2, \alpha}$, for some $ \alpha \in (0,1)$. Then there exists constants $\rho, \gamma>0$, depending only on $\Omega$, such that 
\begin{equation}\label{boundary}
\Vert u \Vert_{L^{\infty}(\Omega_\rho)} \leq \frac{1}{\gamma} \Vert u \Vert_{L^1(\Omega)}
\end{equation}
where $\Omega_\rho := \{ \, x \in \Omega \, : \, dist(x,\partial \Omega) < \rho\, \}$. In particular, $ u \in C^{2,\alpha}(\Omega_\rho \cup \partial \Omega)$. 
\item { If $1\le N\le 9$ and either $\Omega$ is $C^{2,\alpha}$ and convex or $f$ is nondecreasing and $\Omega$ is $C^3$, then there exists constants $\rho, \gamma>0$, depending only on $\Omega$, such that \eqref{boundary} holds.}
\item Fix $N\ge 11$. For every sequence $(\rho_n)\subset\R_+^*$ converging to zero, there exists a sequence of bounded $C^1$ convex domains $\Omega^n\subset\R^N$, $n\in\N^*$, such that the corresponding stable solution $u_n$ to \eqref{equazione} with $f(u)=2(N-3)e^u$ and $\Omega=\Omega^n$ satisfies
$$
\Vert u_n \Vert_{L^{\infty}({(\Omega^n)}_{\rho_n})}\to+\infty\quad\text{yet}\quad \frac1{\vert\Omega^n\vert}\Vert u_n \Vert_{L^{1}{(\Omega^n)}}\quad\text{remains bounded.}
$$
    \end{enumerate}
\end{theorem}

{ The last point of the above theorem shows that in dimension $N\ge 11$, no universal {\it a priori} estimate of the type \eqref{boundary} can hold near the boundary if the domain $\Omega$ is merely convex and the constant $\gamma$ depends on the dimension $N$ and the volume $\vert\Omega\vert$ only. The case $N=10$ is open. We do not know either if locally stable solutions are smooth near the boundary of convex domains, although the universal {\it a priori} estimate fails.}
\medskip

When $\Omega$ is bounded and rotationally invariant we can prove the following classification result. 

\medskip
 
 \begin{theorem}\label{th:simmetria-radiale-limitato} 

 \begin{enumerate}
 \item Let $R>0, N \ge 1,$ $B$ be the open ball $B(0,R) \subset \R^N$ and let $f\in C^1([0,+\infty))$ be a convex function. Assume that $u\in H^1_0(B)$, $f(u)\in L^1_{loc}(B)$ and that $u$ is a stable solution to  
 	\begin{equation}\label{eq:ball}
 	\left\{
 	\begin{aligned}
 	-\Delta u& = f(u) \qquad\text{in $\mathcal D'(B)$}\\
 	u&\geq 0 \qquad\text{a.e. on $B$.}
 	\end{aligned}
 	\right.
 	\end{equation}
 	Then, either $u\equiv 0$ or $u \in C^3(\overline{B} \setminus \{0\})$,$\, u>0$ and $u$ is radially symmetric and radially strictly decreasing. Furthermore, if $ N \leq 9$ and $f$ is nonnegative, then $u \in C^2(\overline{B})$.
 	\item Let $R>0, 1 \leq N \leq 9,$ $B$ be the open ball $B(0,R) \subset \R^N$ and let $f\in C^1([0,+\infty))$ be a nonnegative convex function.
 	Assume that $u\in H^1_0(B)$, $f(u)\in L^1_{loc}(B)$ and that $u$ solves  
 	\begin{equation}\label{eq:morse-rad}
 	\left\{
 	\begin{aligned}
 	-\Delta u& = f(u) \qquad\text{in $\mathcal D'(B)$}\\
 	u&\geq 0 \qquad\text{a.e. on $B.$}
 	\end{aligned}
 	\right.
 	\end{equation}
 	If $u$ has finite Morse index, then either $u\equiv 0$ or $u \in C^2(\overline{B})$, $\, u>0$ and $u$ is radially symmetric and radially strictly decreasing. 
 	\item Let $N \ge 2,$ $\Omega \subset \R^N$ be an open annulus centered at the origin and let $f\in C^1([0,+\infty))$ be a convex function. Assume that $u \in H^1_0(\Omega)$, $f(u)\in L^1_{loc}(\Omega)$ and that $u$ is a stable solution to  
 	\begin{equation}\label{eq:anello}
 	\left\{
 	\begin{aligned}
 	-\Delta u& = f(u) \qquad\text{in $\mathcal D'(\Omega)$}\\
 	u&\geq 0 \qquad\text{a.e. on $\Omega.$}
 	\end{aligned}
 	\right.
 	\end{equation}
 	Then, either $u\equiv 0$ or $u \in C^2(\overline{\Omega})$, $\, u>0$ and $u$ is radially symmetric.\footnote {Furtheremore, if $ \Omega$ is the annulus  $ \{ x \in \R^N \, : \, 0< a < \vert x \vert < b\}$ and $ u(x)=v(\vert x \vert),$ then there is a unique $r_0 \in (a,b)$ such that $ v'>0$ in $(a,r_0), $ $v'(r_0) =0$ and $v' <0$ in $(r_0,b).$ The result follows as in the proof of item 1). For this reason we omit it.}
 	\end{enumerate}
 \end{theorem}
 
 \medskip

\begin{remark}
\begin{enumerate} 
	\item The conclusion that $u$ is radially symmetric in item (1) of Theorem \ref{th:simmetria-radiale-limitato} was already known to hold true in the special case where $u\in C^2(\super B)$ (and with no additional sign assumption on $u$), see e.g. \cite{ab}.
	\item Item (2) of Theorem \ref{th:simmetria-radiale-limitato} is sharp. Indeed, for $N\ge10$, $f(u)=2(N-2)e^u$ and $\Omega=B_1 =B(0,1)$, $u(x)=-2\ln\vert x\vert$ is a singular stable solution in $H^1_0(B_1)$. The finite Morse index assumption is also essential: for $N=3$ and $f(u)=2e^u$, there exists a family of nonradial singular solutions in $\Omega=B_1$ of the form $u(x)=-2\ln\vert x-x_0\vert+v(x)$, where $x_0\neq 0$ and $v\in L^\infty(B)\cap H^1(B)$, see (in details) the proofs in \cite{rebai}. In particular, $u\in H^1_0(B)$. It follows from our result that $u$ cannot have finite Morse index.
	In contrast, note that, for $3 \leq N \leq 9,$ $ \Omega= B_1$ and $ f(u)= 2(N-2) e^u,$ there exist infinitely many smooth and positive solutions to \eqref{equazione} such that $u=0$ on $\partial\Omega$ and with finite and non-zero Morse index. See chapter 2 of \cite{dup} (and the references therein) for a detailed discussion of this topic.
	\item It will be clear from the proof that :
	         \begin{enumerate}
	         \item the radial symmetry in item (1) is still true if we replace $u \in H^1_0(\Omega)$ by any member $u$ of $ H^1(\Omega)$ having constant trace $c \geq 0$ on $ \partial B$. 
	         \item the radial symmetry in item (3) is still true if we replace $u \in H^1_0(\Omega)$ by any member $u$ of $ H^1(\Omega)$ having a nonnegative constant trace on each of the two connected components of the boundary of the annulus (possibly with differents values on the two connected component).
             \item Note that, if $c>0$, we do not claim any monotonicity or special property about the radial profile (as it happens when $ c=0$.)
              \end{enumerate} 
\end{enumerate}
\end{remark}

\medskip


A crucial step in the proof of the above Theorem is the following general symmetry result. As we shall see below, this result also enables us to prove further symmetry results for stable solutions in "symmetric" bounded domains.  

\medskip

\begin{proposition}\label{invariant}
	Assume $N\geq 1$ and let $f\in C^1([0,+\infty))$ be a convex function. Let $\rho \in O(N)$ and let $\Omega \subset \R^N $ be a $\rho-$invariant bounded domain, i.e., a bounded domain such that $\rho(\Omega) = \Omega$. Assume that $u\in H^1_c(\Omega)$\footnote{Here $H^1_c(\Omega)$ denotes the subset of $H^1(\Omega)$ whose members take the constant value $c \geq 0$ on $\partial \Omega$. That is, $H^1_c(\Omega) = \{ u \in H^1(\Omega) \, : \, u - c \in H^1_0(\Omega) \}.$ In particular, for $c=0,$ that set boils down to $H^1_0(\Omega)$. }, $f(u)\in L^1_{loc}(\Omega)$ and that $u$ is a stable solution to  
	\begin{equation}
	\left\{
	\begin{aligned}
	-\Delta u& = f(u) \qquad\text{in $\mathcal D'(\Omega)$}\\
	u&\geq 0 \qquad\text{a.e. on $\Omega$.}
	\end{aligned}
	\right.
	\end{equation}
	If $\rho$ has a fixed point in $ \Omega,$ then $u$ is $\rho$-invariant, namely, $ u(x) = u (\rho x)$ for almost every $x \in \Omega$. 
\end{proposition}

\medskip

\begin{remark}
Below we provide a (non-exhaustive) list of bounded domains to which the above result applies :
\begin{enumerate}
\item if $\Omega$ is any bounded domain symmetric with respect to a hyperplane, then $u$ inherits the same symmetry. 
\item if $\Omega$ is an open ball minus its center $x_0,$ then $u$ is radially symmetric with respect to $x_0$. 
\item if $\Omega$ is a $n$-sided regular polygon,  with $n \geq 3,$ then $u$ is invariant with respect to the  dihedral group $D_n$ (of order $2n$).
\item if $\Omega$ is the product of rotationally invariant bounded domains, i.e., $ \Omega = \omega_1 \times ... \times \omega_m, $ where $\omega_j$ is a bounded rotationnaly invariant domain\footnote{An open ball, an open ball minus its center or an annulus.} of $\R^{n_j}, $ with $n_j \geq 1$ and $ N = n_1 +...+n_m, $  then $u$ inherits the same symmetry, i.e., $u(x)= v(\vert x^1\vert,..., \vert x^m\vert)$ a.e. in $\Omega$.\footnote{Here, for any $j \in \{1,...,m\}$, $x^j$ denotes a generic point of $\omega_j \subset \R^{n_j}$ in such a way that $x := (x^1,...,x^m) \in \R^N$.} This case was already addressed under the additional assumption that $u$ is smooth in Remark 2.1 in \cite{cabre-ros}.
\item if $\Omega$ is a cylinder with $\rho$-invariant cross section, i.e., $\Omega = \omega \times U,$ where $ \omega$ is a $\rho$-invariant bounded domain of $\R^k$, $1 \leq k \leq N-1$ and $U$ is a domain of $\R^{N-k},$ then $u(x)=u(\rho(x^1), x_{k+1},...,x_N)$ a.e. on $\Omega \, ($here $x^1 :=(x_1,...,x_{k}) \in \R^k. )$
\item any bounded domain of "revolution".
\end{enumerate}
\end{remark}

\section{Proofs}

\medskip

\noindent{\bf Proof of Theorem \ref{Prop-approx-crucial}.}

We distinguish two case : either $f'(t) \leq 0$ for any $t > 0$ or there exists $\bar t>0$ such that $ f'(\bar t)>0$.
In the first case, by convexity of $f,$ we have that $ f'(0) \leq f'(t) \leq 0$ for any $t \geq 0,$ then $f$ is also globally Lipschitz-continuous on $[0,+\infty)$. So $f(u) \in L^2(\Omega)$ and $u \in C^2(\Omega)$ by standard elliptic estimates (plus bootstrap and Sobolev imbedding). The claim follows by taking $f_k =f$ and $u_k =u$ for any integer $k\geq 1$. 

In the second case, the convexity of $f$ implies the existence of $t_0 > \bar t$ that $f(t), f'(t) >0$ for any $t \geq t_0$.  Set $k_0 := \lfloor t_0 \rfloor + 1$ (here by $\lfloor t_0\rfloor$ we denote the integer part of $t_0$) and, for any integer $k \geq k_0$ and $t\geq 0$, we set  
\begin{equation}\label{def-approxf}
f_k(t) := \begin{cases} 
f(t)& \,\, \mbox{if} \quad  t \leq k, \\
f(k) + f'(k)(t-k)& \, \, \mbox{if} \quad t > k.
\end{cases} 
\end{equation}	
Clearly, $f_k$ is a convex function of class $ C^1([0,+\infty)) \cap C^{0,1}([0,+\infty))$ and $ f_k \nearrow f$ pointwise in $[0; +\infty)$ (recall that $f(t),f'(t) >0$ for any $t \geq t_0$). Moreover we have
\begin{equation}\label{controllo-derivata-approx}
f_k'(t) \leq f'(t) \qquad \forall \, k \geq k_0, \quad \forall t \geq 0
\end{equation} 
and
\begin{equation}\label{controllo-da-sotto-approxf}
f_k(t) \geq \min_{t \in [0,k_0]} f(t) := c_o(f), \qquad \forall \, k \geq k_0, \quad \forall t \geq 0.
\end{equation} 

In particular, if $f$ is nonnegative, then any function $f_k$ is nonnegative too.
	
Since any $f_k$ is globally Lipschitz-continuous on $[0,+\infty),$ we can use the (standard) method of sub and supersolution in $H^1$ to obtain a stable weak solution	to \eqref{equazione-k} satisfying \eqref{convergence}. To this end we observe that $u \in H^1(\Omega)$ is a nonnegative weak supersolution to 
\begin{equation}\label{approx-weak-problem}
\left\{
\begin{aligned}
-\Delta&v_k = f_k(v_k) \quad \textit{in} \quad \Omega,\\
&v_k- u \in H^1_0(\Omega),\\
\end{aligned}
\right.
\end{equation}
since $ f_k \leq f$ on $[0,+\infty)$ implies that $f_k(u) \varphi \leq f(u) \varphi $ a.e. in $\Omega,$  for any 
nonnegative $ \varphi \in C^{\infty}_c(\Omega)$. So
\begin{equation}
\int_{\Omega} f_k(u) \varphi \leq \int_{\Omega} f(u) \varphi = \int_{\Omega} \nabla u \nabla \varphi \qquad \forall  \varphi \in C^{\infty}_c(\Omega), \quad \varphi \geq 0 \quad {\text {in}} \,\, \Omega
\end{equation}
by \eqref{equazione} and then the above inequality holds true for any nonnegative $\varphi \in H^1_0(\Omega)$ by a standard density argument.  Also, $0$ is a weak subsolution to \eqref{approx-weak-problem}, since $ f(0) \geq 0$ by assumption and $ (0-u)^+ \equiv 0 \in H^1_0(\Omega)$.  Since $ 0 \leq u$ a.e. in $\Omega$, the  method of sub and supersolution in $H^1$ provides a weak solution $v_k$ to \eqref{approx-weak-problem} such that $ 0 \leq v_k \leq u$ a.e. in $\Omega$ and which is \textit{minimal}\footnote{Just consider the solution obtained by the standard monotone iterations procedure starting from the subsolution $\underline u \equiv 0.$} in the following sense : given any weak supersolution $\overline u \in H^1$ of \eqref{approx-weak-problem} such that $ 0 \leq \overline u \leq u$ a.e. in $\Omega,$ we have that $ v_k \leq \overline u$ a.e. in $\Omega.$  From the latter property we immediately infer that $ v_k \leq v_{k+1} \leq u $ a.e. in $\Omega$.  Also, by standard elliptic estimates we have 
$v_k \in C^{2,\alpha}_{loc}(\Omega)$ for any $\alpha \in (0,1).$ By the convexity of $f_k$ and \eqref{controllo-derivata-approx} we see that $f'_k(v_k) \leq f'_k(u) \leq f'(u) $ a.e. in $\Omega$, hence 
\begin{equation}
\int_{\Omega} f'_k(v_k) \varphi^2 \leq \int_{\Omega} f'(u) \varphi^2 \leq  \int_{\Omega} \vert \nabla \varphi \vert^2 \qquad \forall  \varphi \in C^{1}_c(\Omega)
\end{equation} 
and so $v_k$ is a stable weak solution to \eqref{equazione-k}. 

To prove \eqref{convergence} we test \eqref{approx-weak-problem} with $ u-v_k \in H^1_0(\Omega)$ to obtain
\begin{equation}
\int_{\Omega} \nabla v_k \nabla(u-v_k) = \int_{\Omega} f_k(v_k)(u-v_k)
\end{equation}
and so
\begin{equation}
\int_{\Omega} \vert \nabla v_k \vert^2 = \int_{\Omega} \nabla v_k \nabla u - \int_{\Omega} f_k(v_k)(u-v_k) \leq \int_{\Omega} \nabla v_k \nabla u - c_0(f) \int_{\Omega}(u-v_k)
\end{equation}
thanks to \eqref{controllo-da-sotto-approxf} and $u-v_k \geq 0$ a.e. in $\Omega$. Then
\[
\int_{\Omega} \vert \nabla v_k \vert^2 \leq \frac{1}{2}\int_{\Omega} \vert \nabla v_k \vert^2 +  \frac{1}{2}\int_{\Omega}\vert \nabla u \vert^2  + \vert c_0(f) \vert \int_{\Omega}u 
\]
hence
\begin{equation}\label{stima-energie1}
\Vert \nabla v_k \Vert^2_{L^2(\Omega)} \leq \Vert \nabla u \Vert^2_{L^2(\Omega)} + 2 \vert c_0(f) \vert 
\Vert u \Vert_{L^1(\Omega)}. 
\end{equation} 
We also have 
$\Vert v_k \Vert_{L^2(\Omega)} \leq  \Vert u \Vert_{L^2(\Omega)}$
 since $ 0 \leq v_k \leq u$ a.e. in $\Omega$.  Then $(v_k)$ is bounded in $H^1(\Omega)$ and therefore we may and do suppose that (up to subsequences) $v_k \rightharpoonup v$  in $H^1(\Omega)$, $v_k \longrightarrow v $ in $L^2(\Omega)$, $ v_k \nearrow  v$ a.e. on $\Omega$ for some $v \in H^1(\Omega)$, as $k\to \infty$ (actually the whole original sequence converges a.e. to $v,$ since we already know that it is monotone nondecreasing). In particular we have $ 0 \leq v \leq u$ a.e. in $\Omega$. Also note that $v-u \in H^1_0(\Omega),$ since $ v_k -u \rightharpoonup v-u$ in $H^1_0(\Omega)$. 

To prove that $v_k \longrightarrow v$  in $H^1(\Omega)$ we test  \eqref{approx-weak-problem} with $ v-v_k \in H^1_0(\Omega)$ to get
\begin{equation}
\int_{\Omega} \nabla v_k \nabla(v-v_k) = \int_{\Omega} f_k(v_k)(v-v_k)
\end{equation}
and so
\begin{equation}
\int_{\Omega} \vert \nabla v_k \vert^2 = \int_{\Omega} \nabla v_k \nabla v - \int_{\Omega} f_k(v_k)(v-v_k) \leq \int_{\Omega} \nabla v_k \nabla v + \int_{\Omega} \vert c_0(f) \vert(v-v_k)
\end{equation}
thanks to \eqref{controllo-da-sotto-approxf} and $v-v_k \geq 0$ a.e. in $\Omega$. Recalling that $ 0 \leq v_k \leq v$ a.e. on $\Omega$ we then have 
\begin{equation}
\int_{\Omega} \vert v_k \vert^2 + \vert \nabla v_k \vert^2  \leq   \int_{\Omega} \vert v \vert^2 + \int_{\Omega} \nabla v_k \nabla v + \int_{\Omega} \vert c_0(f) \vert(v-v_k)
\end{equation}
and so
\begin{equation}\label{cvforteH1}
\limsup \Vert v_k \Vert_{H^1}  \leq \Vert v \Vert_{H^1} 
\end{equation}
since $v_k \rightharpoonup v$ in $H^1(\Omega)$ and $v_k \to v$ in $L^2(\Omega)$. Inequality \eqref{cvforteH1} and $v_k \rightharpoonup v$ in $H^1(\Omega)$ imply the strong convergence in $H^1(\Omega)$.  

To proceed further we note that $f(v) \in L^1_{loc}(\Omega)$ since
\begin{equation}\label{domina-f(v)}
c_0(f) \leq f_k(v) \leq f(v) = f(v) {\bf 1}_{\{ v \leq k_0 \}} +  f(v) {\bf 1}_{\{ v > k_0 \}} \leq \sup_{t\in [0,k_0]} f(t) + f(u),
\end{equation}
where in the latter we have used that $f'(t)>0$ for $t \geq k_0$ and $ v \le u $ a.e in $\Omega$. 

Next we prove that $ -\Delta v = f(v) $ in $\mathcal D'(\Omega)$. This follows by passing to the limit in \eqref{approx-weak-problem} by the Lebesgue's dominated convergence theorem, after having observed that $f_k(v_k) \longrightarrow f(v)$ a.e. in $\Omega,$ 
and $ c_0(f) \leq f_k(v_k) \leq f(v_k) = f(v_k) {\bf 1}_{\{ v_k \leq k_0 \}} +  f(v_k) {\bf 1}_{\{ v_k > k_0 \}} \leq \sup_{t\in [0,k_0]} f(t) + f(u)$ holds true.

So far we have proved that 
\begin{equation}\label{equazione-v}
\left\{
\begin{aligned}
-\Delta&v = f(v) \qquad\text{in $\mathcal D'(\Omega)$}\\
&v- u \in H^1_0(\Omega),\\
&0 \leq v\leq u \qquad\text{a.e. on $\Omega$},
\end{aligned}
\right.
\end{equation}
and therefore from Theorem \ref{theorem:compa-stable-solutions} we deduce that either $u=v$ in $\Omega$ or $f(t) = a + \lambda_1 t$ for all $t \in (\inf_{\Omega} v, \, \sup_{\Omega} u)$ and some $a \in \R,$ $u,v \in C^\infty(\Omega)$ and $u-v$ is a positive first eigenfunction of $ -\Delta$ with homogeneous Dirichlet boundary conditions. In the first case we are done, while in the second one we distinguish two subcases : either $\sup u = +\infty$ or not. In the first subcase $f$ is necessarily globally Lipschitz-continuous on $[0, +\infty)$ and the conclusion follows by taking $f_k=f$ and $u_k =u$ for every $k$. If  $\sup u < +\infty$, the conclusion follows by taking the sequences $(f_k)_{k \geq \bar k} $ and $(u)_{k \geq \bar k} $, where $\bar k$ is any integer satisfying $\bar k > \sup u$.



\bigskip

\noindent{\bf Proof of Theorem \ref{smoothness}. }

When $N=1,$ any  $u\in H^1_{loc}(\Omega)$ is continuous by Sobolev imbedding and so is $f(u)$. This implies $u \in C^2$ and $f(u) \in C^1$, so $u$ is of class $C^3$ by using equation \eqref{equazione}.
So assume that $N\ge2$ and that $u$ is locally stable. For any point $x_0 \in \Omega$ pick a ball $B(x_0,r_0) \subset \Omega$ in which $u$ is stable and set $B= B(x_0,r_0)$. 

By Theorem \ref{Prop-approx-crucial} there is a nondecreasing sequence $(f_k)$ of functions in $C^1([0,+\infty)) \cap C^{0,1}([0,+\infty))$ such that $ f_k \nearrow f$ pointwise in $[0; +\infty)$ and a nondecreasing sequence $(u_k)$ of functions in $H^1(\Omega) \cap C^2(\Omega)$ such that $u_k$ is a stable weak solution\footnote{That is a function $u_k$ satisfying $\int_\Omega \nabla u_k \nabla \varphi = \int_\Omega f_k(u_k) \varphi$, for all $ \varphi \in H^1_0(\Omega)$} to \eqref{equazione-k} such that \eqref{convergence} holds. 

Since $ N \leq 9$ an application of  Theorem \ref{th:cfrs} to $u_k$ yields 
 \[
 \Vert u_k \Vert_{C^{\beta}(\overline{B(x_0, \frac{r_0}{2}}))} \leq C \Vert u_k \Vert_{L^1(B)} \leq C \vert B \vert^{\frac{1}{2}} 
 \Vert u_k \Vert_{L^2(B)} \leq C'\Vert u \Vert_{L^2(B)}
 \]
 where $\beta \in (0,1)$ depends only on $N$, while $C>0$ depends on $N$ and $r_0$. 
 In particular, the sequence $(u_k)$ is bounded in $C^{\beta}(\overline{B(x_0, \frac{r_0}{2}))}$ and the Ascoli-Arzel\`a's theorem then implies that a subsequence must converge uniformly to some continuous function $v$ on $\overline{B(x_0, \frac{r_0}{2}))}$. This function must coincide with $u$ on $\overline{B(x_0, \frac{r_0}{2}))}$, since we already know that $u_k \longrightarrow u$ a.e. on $\Omega$. The boundedness of $u$ on $B(x_0, \frac{r_0}{2})$ and standard elliptic theory imply $u \in C^{2,\alpha}_{loc}(B(x_0, \frac{r_0}{2}))$ for every $ \alpha \in (0,1)$. 
 Since $x_0$ is an arbitrary point of $ \Omega$, this concludes the proof. 

  \hfill\qed 
 
 \bigskip
 
 \noindent{\bf Proof of Theorem \ref{th:simmetria-radiale-limitato}.}

1) In is enough to treat the case $ u \not \equiv 0$. We first prove that $u$ is radially symmetric. For $\rho \in O(N)$ we set $ u_\rho(x):= u(\rho x)$, $ x \in B$. Then $u_\rho \in H^1_0(B)$ is a stable solution to \eqref{eq:ball} since $u$ is so. We can therefore apply Theorem \ref{theorem:compa-stable-solutions-nonordered} below to get that $u$ and $u_\rho$ are ordered solutions. 
If $u \not \equiv u_\rho$, an application of Theorem \ref{theorem:compa-stable-solutions} would give $u, u_\rho \in C^{\infty} $ and, either $u < u_\rho $ or $ u > u_\rho$ in $B$. The latter are clearly impossible, since $ u(0) = u_\rho(0)$. Thus, $u\equiv u_\rho$ for any $\rho \in O(N)$, and so $u$ is radially symmetric in $B$. Since $u$ is a radially symmetric member of $H^1_0(B),$ we have that $u \in C^0(\overline{B}\setminus \{0\})$ and by standard elliptic regularity, $u \in C^{2,\alpha}_{loc}((\overline{B}\setminus \{0\})$, for any $\alpha \in (0,1)$. Hence we can write $u(x) = v(r),$ $r = \vert x \vert  \in (0,R],$ and  so $ v \in C^2((0,R])$ is a classical solution to the ode $ - (r^{N-1} v')' =r^{N-1}f(v) $ in $(0,R].$ The latter clearly implies $ v \in C^3((0,R])$ but also that\footnote{In view of \eqref{segno-derivata-v} we could have used the moving planes procedure to get the strict monotonicity of 
$u$ in the radial direction. However, we have chosen to give an elementary proof of this fact, which highlights the role played by the stability assumption on $u$. }
\begin{equation}\label{segno-derivata-v}
v' <0 \quad \text {in} \quad (R-\epsilon, R),
\end{equation}
for some $\epsilon \in (0,R)$. Indeed, if $f(0)\geq0$ the Hopf's lemma yields $ v'(R)<0$ (recall that we are supposing that $ u \not \equiv 0)$ and so \eqref{segno-derivata-v} follows. When $f(0)<0$ we have $ (r^{N-1} v')' = - r^{N-1}f(v) > 0$ in an interval of the form $(R-\epsilon,R),$ thus $r \to  r^{N-1} v'$ is strictly increasing in  $(R-\epsilon, R)$. The latter and the fact that  $v'(R) \leq 0$ (recall that $v \geq 0$ in $(0,R)$ and $v(R)=0$) imply \eqref{segno-derivata-v}. 

To conclude it is enough to prove that $v'<0$ on $(0,R)$ (this also implies that $u>0$ in $B\setminus \{0\}$). 
Suppose not, then 
\begin{equation}\label{def:r0}
r_0 := \inf \{ r \in (0,R) \, : \, v' < 0 \,\, \text{on} \,\, (r,R) \}.
\end{equation}
is well-defined, $r_0$ belongs to $(0,R)$ and $v'(r_0) =0$. We have two cases : either there exists $z \in (0,r_0)$ such that $v'(z)=0,$ or $v'$ has a sign on $(0,r_0)$ (i.e., either $v'<0$ or $v'>0$ on $(0,r_0)$). Let us show that both of them are impossible. In the first case we observe that $w := u_r  $, the radial derivative of $u$, is of class $C^2_0(\overline{A_{z,r_0}})$ and satisfy 
\begin{equation}\label{lin-radial-derivative}
- \Delta w + (N-1) \frac{w}{r^2}= f'(u)w \quad \text{on} \quad A_{z,r_0},
\end{equation}
where $A_{z,r_0} := \{ x \in \R^N \, : \, z < \vert x \vert < r_0 \}.$  We can then multiply \eqref{lin-radial-derivative} by $w$, integrate by parts and find 
\begin{equation}\label{ug-int1}
\int_{A_{z,r_0}} \vert \nabla w \vert^2 - \int_{A_{z,r_0}} f'(u)w^2 = - \int_{A_{z,r_0}} (N-1) \frac{w^2}{r^2}.
\end{equation}
On the other hand $w \in C^2_0(\overline{A_{z,r_0}})$ and $u$ is stable on $A_{z,r_0}$, so $w$ can be used as test function in the stability condition satisfied by $u$ to obtain 
\[
\int_{A_{z,r_0}} \vert \nabla w \vert^2 - \int_{A_{z,r_0}} f'(u)w^2 \geq 0. 
\]
The latter and \eqref{ug-int1} give $w=0$ on $A_{z,r_0}$ and so $ w \leq 0$ in the open annulus $A_{z,R}$. Since 
$w$ solves the linear equation $- \Delta w + ( \frac{(N-1)}{r^2} -f'(u)) w =0 $ on $A_{z,R},$ the strong maximum principle implies $ w \equiv 0$ on $ A_{z,R}$, and so $v' \equiv 0$ on $(z,R),$ which contradicts \eqref{segno-derivata-v}. In the second case, if $v'<0$ on $(0,r_0)$ then $w \leq 0$ on $A_{0,R}$ and, as before, $ w \equiv 0$ on $ A_{0,R}$ by the the strong maximum principle. The latter is again in contradiction with \eqref{segno-derivata-v}. Hence we are left with $v'>0$ on $(0,r_0)$. In this case, the latter and the definition of $r_0$ imply $0 \leq u \leq v(r_0)$ on $B \setminus \{0\},$ so $u \in L^{\infty}(B)$ and then $u \in C^2(\overline B)$ by standard elliptic estimates. To achieve a contradiction, we first observe that $ - \Delta w + (N-1) \frac{w}{r^2}= f'(u)w $
on the annulus $A_{r,r_0} := \{ x \in \R^N \, : \, r < \vert x \vert < r_0 \}$, for every $r \in (0,r_0)$, we then multiply the equation by $w$, integrate by parts to get 
\begin{equation}\label{ug-int2}
\int_{A_{r,r_0}} \vert \nabla w \vert^2 - \int_{A_{r,r_0}} f'(u)w^2 + \int_{A_{r,r_0}} (N-1) \frac{w^2}{r^2} = \int_{\partial A_{r,r_0}} \frac{\partial w}{\partial \nu} w.
\end{equation}
Note that $\int_{\partial A_{r,r_0}} \frac{\partial w}{\partial \nu} w \longrightarrow 0$, as $r \to 0^+$. Indeed, the function
\begin{equation}\label{def:w-tilde}
\tilde w(x): = \begin{cases}
w(x) & \text {if} \quad x \in B \setminus \{0\},\\
0 & \text {if} \quad x=0, 
\end{cases}
\end{equation}
is of class $C^{0,1}(B)$
\footnote{Since $u \in C^2(\overline B)$ is a radial function we have $\nabla u (0)=0,$ $ w := u_r \in C^1(\overline B \setminus \{0\})$ and $ \vert w (x) \vert \le C \vert x \vert$ in $\overline B $, for some constant $C>0$. Then $ \tilde w \in C^0(\overline B)$, $\vert v'(r) \vert \leq  Cr $ in $[0,R]$ and so $\vert v''(r) \vert \leq C_1$ in $(0,R]$, for some constant $C_1>0,$ by using the ode satisfied by $v$. Thus, for any $x,y \in B$ such that $ \vert y \vert > \vert x \vert >0$ we have : $ \vert w(y)-  w(x) \vert = \vert v'(\vert y \vert )- v'(\vert x \vert) \vert \leq \int_{\vert x \vert}^{\vert y \vert} \vert v''(t) \vert \,dt \leq \sup_{\xi \in \left[ \vert x \vert, \vert y \vert \right]} \vert v''(\xi) \vert \left(\vert y \vert - \vert x \vert \right) \leq \sup_{r \in B \setminus \{0\} } \vert v''(r) \vert y-x \vert \leq C_1 \vert y-x \vert $. Then, $\tilde w \in C^{0,1}(B).$}
and therefore $ \vert \int_{\partial A_{r,r_0}} \frac{\partial w}{\partial \nu} w \vert \leq  \int_{S_r} \vert \nabla \tilde w \vert \vert w \vert \leq \Vert \nabla \tilde w \Vert_{L^\infty(B)} \sigma(S_r)v'(r) \longrightarrow 0$ (here $S_r$ is the sphere $ \{ x \in \R^N \, : \, \vert x \vert = r \}$ and  $\sigma(S_r)$ denotes its measure). Also observe that 
\[
\int_{A_{r,r_0}} \vert \nabla w \vert^2 - \int_{A_{r,r_0}} f'(u)w^2 + \int_{A_{r,r_0}} (N-1) \frac{w^2}{r^2} \longrightarrow \int_{B(0,r_0)} \vert \nabla \tilde w \vert^2 - \int_{B(0,r_0)} f'(u) \tilde w^2 + \int_{B(0,r_0)} (N-1) \frac{\tilde w^2}{r^2} 
\]
by monotone and dominated convergence theorems. By gathering together all those information we are
led to 
\begin{equation}\label{ug-int3}
\int_{B(0,r_0)}\vert \nabla \tilde w \vert^2 - \int_{B(0,r_0)}f'(u)\tilde w^2 = - \int_{B(0,r_0)}(N-1) \frac{\tilde w^2}{r^2}.
\end{equation}
Finally, since $u$ is stable on $B$, then $u$ is stable on $B(0,r_0)$ too. Thanks to item i) of Proposition \ref{prop:res-auxiliares}) we can use $w \in C^{0,1}_0(\overline {B(0,r_0)}) \subset H^1_0(B(0,r_0))$ as test function in the stability condition satisfied by $u$ to get $\int_{B(0,r_0)} \vert \nabla \tilde w \vert^2 - \int_{B(0,r_0)} f'(u)\tilde w^2 \geq 0$. As before, the latter and  \eqref{ug-int3} imply that $w \equiv 0$ on $B(0,r_0) \setminus \{0\}$, which is impossible since we are assuming that $v'>0$ in $(0,r_0).$ 

\smallskip

If $ N \leq 9$ and $f \geq 0,$ then $u$ is $C^2$ in a neighborhood of the origin by item (1) of Theorem \ref{smoothness}. This completes the proof of item (1). 

\medskip

2) By combining item (1) and item (2) of Theorem \ref{smoothness} we have that $u \in C^2(\overline{B})$. Then, either $u \equiv 0$ or $ u>0$ in $B$ by the strong maximum principle. The desired  conclusion then follows by a celebrated result of Gidas, Ni and Nirenberg.

\medskip

3) Let $H$ be any hyperplane through the origin and let $ \rho=\rho(H) \in O(N)$ be the corresponding reflection with respect to $H$.  Set $ u_\rho(x):= u(\rho x)$, $ x \in \Omega$. Then $u_\rho \in H^1_0(\Omega)$ is a stable solution to \eqref{eq:anello} since $u$ is so. By Theorem \ref{theorem:compa-stable-solutions-nonordered}, $u$ and $u_\rho$ are ordered solutions. If $u \not \equiv u_\rho$, an application of Theorem \ref{theorem:compa-stable-solutions} would give $u, u_\rho \in C^{\infty} $ and, either $u < u_\rho $ or $ u > u_\rho$ in $\Omega$. The latter are clearly impossible, since $ u=u_\rho$ on the hyperplane $H$. Thus, $u\equiv u_\rho$ and so $u$ is radially symmetric in $\Omega$, since $H$ is arbitrary. Since $u$ is a radially symmetric member of $H^1_0(\Omega),$ we have that $u \in C^0(\overline {\Omega})$ and standard elliptic theory imply 
$u \in C^{2,\alpha}_{loc}(\overline{\Omega})$, for any $\alpha \in (0,1)$. The required regularity then follows, as in item (1), by analysing the ode satisfied by the radial profile of $u$. \hfill\qed

\section{Some auxiliary results} \label{S3}

In this section we prove some results for stable solutions in $H^1$. These results have been used in the proofs of our main results and some of them are of independent interest.
In what follows, when $\Omega$ is bounded, we shall denote by $ \lambda_1 = \lambda_1(\Omega)>0 $ the principal eigenvalue of $ -\Delta$ with homogeneous Dirichlet boundary conditions.  
\begin{lemma} \label{lemma lambda1}
	Let $\Omega$ be a bounded domain of $\R^N$, $N \geq 1$ and let $f \in C^1(\R)$ such that $f(t) > \lambda_1 t$ for $t>0$. Assume that $u\in H^1(\Omega)$ and  $f(u) \in L^1_{loc}(\Omega)$.
	
 If $u$ solves 
	\begin{equation}
	\left\{
	\begin{aligned}
	-\Delta u& \geq f(u) \qquad\text{in $\mathcal D'(\Omega)$}\\
	u&\geq 0 \qquad\text{a.e. on $\Omega$},
	\end{aligned}
	\right.
	\end{equation}
    then $u \equiv 0$, $f(0)=0$ and $u$ is a solution to 
\begin{equation}\label{sol-non-stabile} 
-\Delta u = f(u) \qquad\text{in $\mathcal D'(\Omega)$}
\end{equation}
which is not stable.
\end{lemma}

\begin{proof} By the strong maximum principle (for superharmonic functions in $H^1$) either $u \equiv 0$ or $ u >0$ a.e. on 
	$\Omega$. Let us prove that the latter is impossible. To this end, let us suppose that $u>0$ and let $ h \in H^1(\Omega)$ be the unique weak solution to 
	\begin{equation}\label{harmonica}
	\left\{
	\begin{aligned}
	-&\Delta h = 0 \quad\text{in $\Omega$,}\\
	& u-h \in H^1_0(\Omega).
	\end{aligned}
	\right.
	\end{equation}
Then $ 0 \leq h \leq u $ a.e. on $\Omega$ by the maximum principle, and $ h < u$ a.e. on $\Omega$ by the strong maximum principle. Indeed, $h=u$ would imply $ 0 = - \Delta(u-h) =f(u) > \lambda_1 u >0$ a.e. on $\Omega$, a contradiction. 

Let $\phi_1 \in H^1_0(\Omega)$ be a positive function associated to the eigenvalue $\lambda_1$. Since $f(u) \geq 0$ a.e. on $\Omega$, a standard density argument and Fatou's Lemma imply 
	
	\[
	\int_{\Omega} \nabla u \nabla \varphi \geq  \int_{\Omega} f(u) \varphi \qquad \forall \varphi \in H^1_0(\Omega), \quad \varphi \geq 0 \,\, \text {a.e. in} \,\, \Omega,
	\]
	and therefore we can use $\phi_1$ in the latter to get
	\[
	\int_{\Omega} \nabla u \nabla \phi_1 \geq  \int_{\Omega} f(u) \phi_1. 
	\]
	Then,
	\[
	\int_{\Omega} \lambda_1 u \phi_1  <  \int_{\Omega} f(u) \phi_1 \leq \int_{\Omega} \nabla u \nabla \phi_1 = \int_{\Omega} \nabla (u-h) \nabla \phi_1 + \int_{\Omega} \nabla h \nabla \phi_1 = \int_{\Omega} \nabla (u-h) \nabla \phi_1 +0
	\]
	where we have used the property of $f$, as well as the fact that $h$ solves \eqref{harmonica}.
	In addition, by definition of $\phi_1$ and $u-h \in H^1_0(\Omega)$, we also have $ \int_{\Omega} \nabla (u-h) \nabla \phi_1 = \int_{\Omega} \lambda_1 (u-h) \phi_1$, which leads to 
	\[
	\int_{\Omega} \lambda_1 u \phi_1  <  \int_{\Omega} f(u) \phi_1 \leq \int_{\Omega} \nabla u \nabla \phi_1 = \int_{\Omega} \nabla (u-h) \nabla \phi_1 = \int_{\Omega} \lambda_1 (u-h) \phi_1 \leq \int_{\Omega} \lambda_1 u \phi_1 
	\]
	a contradiction. Therefore, $u \equiv 0$ and $f(0) \leq -\Delta u =0$. The latter and the assumption on $f$ give $f(0)=0$ and so $u$ solves \eqref{sol-non-stabile}. 
	 Since $f'(0) > \lambda_1 = \inf \, 
	\{ \frac{\int_{\Omega} \vert \nabla u \vert^2}{\int_{\Omega} u^2} \, : \, u \in H^1_0(\Omega), u \not \equiv 0 \} $ we immediately see that $u=0$ cannot be a stable solution to \eqref{sol-non-stabile}.  
	
\hfill\qed

\end{proof}

\begin{proposition}\label{prop:res-auxiliares}Let $\Omega$ be an open set of $\R^N, N \geq 1,$ and let $f \in C^1([0;+\infty))$ be a convex function.
	
\noindent i) Let $u \in L^1_{loc}(\Omega)$ such that $u \geq 0 $ a.e. in $ \Omega$ and
\begin{equation}\label{hyp-stab}
f'(u) \in  L^1_{loc}(\Omega), \qquad \int_{\Omega} f'(u) \phi^2 \leq  \int_{\Omega} \vert \nabla \phi \vert^2
\qquad \qquad \forall \phi \in C^{\infty}_c(\Omega). 
\end{equation}
Then 
\begin{equation}\label{stab-migliorata}
f'(u) \varphi^2 \in L^1(\Omega), \qquad  \int_{\Omega} f'(u) \varphi^2 \leq  \int_{\Omega} \vert \nabla \varphi \vert^2
\qquad \qquad \forall \varphi \in H^1_0(\Omega).
\end{equation}	

\medskip

\noindent ii) Let $u,v \in H^1(\Omega)$ such that $f(u),f(v) \in  L^1_{loc}(\Omega)$, $(u-v)^+ \in H^1_0(\Omega)$ and $ 0 \leq v, \,\, 0  \leq u$ a.e. in $ \Omega$. Also assume that $u$ satisfies \eqref{hyp-stab}. Then, $(f(u) -f(v))(u-v)^+ \in L^1(\Omega).$

\medskip

\noindent iii) Let $u,v \in L^1_{loc}(\Omega)$ such that $f(u),f(v) \in  L^1_{loc}(\Omega)$ and $ 0 \leq v \leq u$ a.e. in $ \Omega$. 

\noindent Then, $f'(v) (u-v) \in L^1_{loc}(\Omega).$ 
\end{proposition}

\begin{proof}
	
i) 	By convexity of $f$ and \eqref{hyp-stab} we have 
\begin{equation}\label{stab-negativa}
f'(0) \leq f'(u) \qquad \text {a.e. in} \,\, \Omega
\end{equation}
and
\begin{equation}\label{stab-spezzata}
\int_{\Omega} [f'(u)]^+ \phi^2 \leq \int_\Omega \vert \nabla \phi \vert^2 + \int_{\Omega} [f'(u)]^- \phi^2 \qquad \qquad \forall \phi \in C^{\infty}_c(\Omega).
\end{equation}
From  \eqref{stab-negativa} we deduce that $[f'(u)]^-  \leq  [f'(0)]^- $ and so $[f'(u)]^- \in L^{\infty}(\Omega)$ and $ [f'(u)]^- \varphi^2 \in L^1(\Omega)$ for any $\varphi \in H^1_0(\Omega)$. Now, for any $\varphi \in H^1_0(\Omega),$
let $ (\phi_n)$ be a sequence of functions in $ C^{\infty}_c(\Omega)$ such that $ \phi_n \longrightarrow
 \varphi$ in $H^1_0(\Omega)$ and a.e. in $\Omega$.  Using $\phi=\phi_n$ in \eqref{stab-spezzata} and Fatou's Lemma we immediately get that \eqref{stab-spezzata} holds true for any $ \varphi \in H^1_0(\Omega)$. In particular, $ [f'(u)]^+ \varphi^2 \in L^1(\Omega)$ and the two claims of \eqref{stab-migliorata} follow. 

ii) The convexity of $f$ and $(u-v)^+ \geq 0$ a.e. on $\Omega$ imply 
\begin{equation}\label{disug-conv-migliorata2}
f'(0)[(u-v)^+]^2 \leq f'(v)[(u-v)^+]^2 \leq (f(u) -f(v))(u-v)^+ \leq f'(u)[(u-v)^+]^2 \qquad \text{a.e. on} \quad\Omega 
\end{equation}
\noindent and so 
\begin{equation}\label{buona-int}
(f(u) -f(v))(u-v)^+ \in L^1(\Omega)
\end{equation}
thanks to $(u-v)^+ \in H^1_0(\Omega)$ and item i).

iii) By convexity of $f$ we have 
\[
f(u)-f(v) \geq f'(v)(u-v) \geq f'(0)(u-v)
\]
which implies the desired conclusion.

\hfill\qed

\end{proof}

\begin{theorem} 
	Let $\Omega$ be a bounded domain of $\R^N$, $N \geq 1$ and let $f \in C^1([0;+\infty))$ be a convex function. Assume that $u,v \in H^1(\Omega)$ satisfy $ u-v \in H^1_0(\Omega),$ $0 \leq v \leq u $ a.e. on $\Omega,$ $f(u),f(v) \in L^1_{loc}(\Omega)$ and both $u$ and $v$ are solution to 
	\begin{equation}\label{eq-D1}
	-\Delta w = f(w)\qquad\text{in $\mathcal D'(\Omega)$}.
	\end{equation}
	If $f'(u) \in L^1_{loc}(\Omega)$ and $u$ is stable, then either $u\equiv v$ or $f(t) = a + \lambda_1 t$ for all $t \in (\inf_{\Omega} v, \, \sup_{\Omega} u)$ and some\footnote{Actually the real number $a$ is unique and its value is given by $- \lambda_1\frac{\int_\Omega \phi_1 h}{\int_\Omega \phi_1} \leq 0, $ where $\phi_1$ is a positive first eigenfunction of $ -\Delta$ with homogeneous Dirichlet boundary conditions and $h \in H^1(\Omega)$ is the unique weak solution of $ - \Delta h =0$ in $ \Omega$ with $ u-h \in H^1_0(\Omega)$ (to see this, use $\phi_1$ as test function in the weak formulation of $-\Delta u = a + \lambda_1 u$ and the fact that $h$ is harmonic) and also note that $h$ is nonnegative by the maximum principle. In particular, $u,v \in H^1_0(\Omega) \Longleftrightarrow a=0$. Also note that, for every $ a \leq 0$ there exist solutions $u,v$ for which the second alternative of the theorem occurs. Indeed, the functions $u_t := - \frac{a}{\lambda_1} + t \phi_1$, $ t \geq 0$ are suitable.}
	$a \in \R,$ $u,v \in C^\infty(\Omega)$ and $u-v$ is a positive first eigenfunction of $ -\Delta$ with homogeneous Dirichlet boundary conditions.   
\end{theorem}

\medskip



\begin{proof} 

By item ii) of Proposition \ref{prop:res-auxiliares} we have 	
\begin{equation}\label{buona-int-2}
(f(u) -f(v))(u-v) \in L^1(\Omega),
\end{equation}	
note that $ u-v = (u-v)^+$ since $v \leq u$ a.e. in $\Omega$. 
	
We claim that 
	\begin{equation}\label{equazione-migliorata}
	\int_{\Omega} \vert \nabla (u-v) \vert^2 =  \int_{\Omega} (f(u)-f(v))(u-v).
	\end{equation}
	To this end, recall that 
	\begin{equation}\label{equazione-approx1}
	\int_{\Omega} \nabla (u-v) \nabla \varphi  =  \int_{\Omega} (f(u)-f(v))\varphi \qquad \forall \varphi \in C^{\infty}_c(\Omega)
	\end{equation}
	by assumption. Then, since $ f(u)-f(v) \in L^1_{loc}(\Omega)$, a standard approximation argument and Lebesgue's dominated convergence theorem yield that \eqref{equazione-approx1} holds true for any $\varphi \in H^1(\Omega) \cap L^{\infty}(\Omega)$ with compact support. Now we prove that \eqref{equazione-approx1} holds true for any 
	$\varphi \in H^1(\Omega)$ with compact support and such that $ 0 \leq \varphi \leq u-v$ a.e. on $\Omega$.  Indeed, we can apply \eqref{equazione-approx1} with $\varphi_n = T_n(\varphi)$, where $T_n = T_n(t)$ denotes the truncation function at level $n \geq 1,$ and write
	\begin{equation}\label{equazione-approx2}
	\int_{\Omega} \nabla (u-v) \nabla \varphi_n  =  \int_{\Omega} (f(u)-f(v))\varphi_n, \qquad \forall n \geq 1.
	\end{equation}	
	Since $ \varphi_n \to \varphi$ in $ H^1_0(\Omega)$ and a.e. on $\Omega,$ and \eqref{buona-int-2} is in force, an application of Lebesgue's dominated convergence theorem we get
	\begin{equation}\label{equazione-approx2}
	\int_{\Omega} \nabla (u-v) \nabla \varphi  =  \int_{\Omega} (f(u)-f(v))\varphi \qquad \forall \varphi \in H^1(\Omega), \quad supp(\varphi) \subset \subset \Omega, \quad 0 \leq \varphi \leq u-v \quad \text {a.e. in} \,\, \Omega.
	\end{equation}
	To conclude the proof of \eqref{equazione-migliorata} we proceed as follows. Let  $(\rho_n)$ be a sequence of nonnegative functions in $C^{\infty}_c(\Omega)$ such that $\rho_n \to u-v $ in $ H^1_0(\Omega)$ and a.e. on $\Omega$ and set $ w_n := \min \{ u-v, \rho_n \}$. Then, $ w_n \in H^1(\Omega)$, $supp(w_n) \subset supp(\rho_n)  \subset \subset \Omega $ and $ 0 \leq w_n \leq u-v$ a.e. on $\Omega$.  We can therefore apply \eqref{equazione-approx2} with $\varphi = w_n$ to get the desired conclusion \eqref{equazione-migliorata}, since 
	$ w_n \to u-v$ in $ H^1_0(\Omega)$ and a.e. on $\Omega,$ and \eqref{buona-int-2} is in force.

	By convexity of $f$ we have $ f'(0) \leq f'(v) \leq f'(u)$ a.e. on $ \Omega$, hence $f'(v) \in L^1_{loc}(\Omega)$ and $ [f'(v)]^- \leq [f'(0)]^-$a.e. in $\Omega$. Also observe that $ - \Delta (u-v) = f(u) -f(v) \geq f'(v)(u-v)$ in $\mathcal D'(\Omega)$ and so also $- \Delta (u-v) + [f'(v)]^-(u-v) \geq [f'(v)]^+(u-v) \geq 0$ in $\mathcal D'(\Omega),$ since $u-v \in H^1_0(\Omega)$ is nonnegative a.e. in $ \Omega$. 
	The strong maximum principle (see e.g. Theorem 8.19 in \cite{GiTr}) yields either $ u \equiv v,$ and we are done, or $ u > v $ a.e. on $\Omega$. In the remaining part of the proof we assume that the latter possibility is in force and we set $I := \inf_{\Omega} v, \, S := \sup_{\Omega}u$. 
	
	Now we combine \eqref{equazione-migliorata} with \eqref{stab-migliorata} with $ \varphi = u-v$ to get
	\[
	\int_{\Omega} \big( f(u)-f(v) - f'(u)(u-v) \big)(u-v) \geq 0. 
	\]
	By convexity of $f$, the integrand in the above inequality is nonpositive and so
	\[
	\big( f(u)-f(v) - f'(u)(u-v) \big)(u-v) = 0 \qquad \text {a.e. in} \,\, \Omega
	\]
	and then
	\[
	f(u)-f(v) - f'(u)(u-v) = 0 \qquad \text {a.e. in} \,\, \Omega
	\]
since $ u>v$ a.e. on $\Omega$. The latter implies that, for almost every $x \in \Omega$, the function $f$ must be affine on the open interval $(v(x), u(x))$.  Now we prove that $f$ is an affine function on the the interval $(I,S)$. To this end we first consider the case in which one of the two solutions is a constant function. If $v$ is constant, say $ v \equiv c $, then $ I = c $ and for almost every $ x \in \Omega $ the function $f$ is affine on the open interval $ (I, u(x))$. Since any two intervals of this form intersect and $f$ is affine on each of them, we see that $f$ must be the same affine function on both intervals. This implies that $ f $ is an affine function over the entire interval $ (I, S),$ as claimed. The same argument applies if $u$ is constant. It remains to consider the case where neither $v$ nor $u$ are constant. Since $v$ is not constant we can find $x_0 \in \Omega$ such that 

\begin{equation}\label{proprieta-x_0}
\left\{
\begin{aligned}
\quad & v(x_0) < u(x_0) \quad \text{and $ \, \, f(t) = a + b t \,\,$ for all $t \in (v(x_0), u(x_0)),$} \\
& \vert \{ \, v > v(x_0) \, \} \vert >0,\\
& \vert \{ \, v \leq v(x_0) \, \} \vert >0,
\end{aligned}
\right.
\end{equation}
where we denoted by $ \vert X \vert$ the Lebesgue measure of any  measurable set $X \subset \Omega$.  

Let $(\alpha, \beta)$ be the largest open interval containing $(v(x_0), u(x_0))$ and such that
$f(t) = a + b t$ for all $t \in (\alpha, \beta)$. To conclude the proof it is enough to show that $(I,S) \subset (\alpha, \beta)$. Let us first prove that $ I \geq \alpha$. Assume to the contrary that $ I < \alpha,$ then the latter and \eqref{proprieta-x_0} give $I < \alpha \leq v(x_0) < \sup_{\Omega} v$ and so we can find an integer $m \geq 1$ such that 
\begin{equation}\label{Omega-rho}
I \leq \inf_{\omega_m} v < \alpha < \sup_{\omega_m} v
\end{equation}
where $\{ \omega_m, m \geq 1 \} $ is a countable family of open connected sets of $\R^N$ such $\omega_m \subset \omega_{m+1} \subset \subset \Omega$ and $ \Omega = \cup_{m \geq 1} \omega_m $ (such a family of open  connected sets exists thanks to the fact that $\Omega$ is open, bounded and connected). Recall that $u-v \in H^1_0(\Omega)$ is positive a.e on $\Omega$ and satisfies $- \Delta (u-v) + [f'(v)]^-(u-v) \geq [f'(v)]^+(u-v) \geq 0$ in $\mathcal D'(\Omega),$ where $ [f'(v)]^- \in L^{\infty}(\Omega)$, then the strong maximum principle yields 
\begin{equation}\label{controllo-armonico}
u - v \geq  c(m)>0 \quad \text {a.e. on $\omega_m$.}
\end{equation}

Now, if we set $ \epsilon := \min \{ \,  \frac{c(m)}{2}, \alpha - \inf_{\omega_m} v\, \} >0$, we have $I \leq \inf_{\omega_m} v < \alpha - \epsilon < \alpha < \sup_{\omega_m} v $ and so 
\begin{equation}\label{controllo-livelli}
\left\{
\begin{aligned}
\quad & \vert \{ \, v \leq \alpha - \epsilon \, \}  \cap \omega_m \vert >0,\\
& \vert \{ \, v \geq \alpha - \frac{\epsilon}{2} \} \cap \omega_m \vert >0.
\end{aligned}
\right.
\end{equation}

Since $\omega_m$ is open and connected we can use the "intermediate value theorem" for functions in a Sobolev space (see e.g. Th\'eor\`eme 1 of \cite{C-Haraux}) to get 
$\vert \{ \, \alpha - \epsilon \leq v \leq \alpha - \frac{\epsilon}{2} \}\, \}  \cap \omega_m \vert >0.$ This fact and $ \eqref{controllo-armonico}$ prove the existence of $x_\epsilon \in \omega_m$ such that $\alpha - \epsilon \leq v(x_\epsilon) \leq \alpha - \frac{\epsilon}{2}$ and $ u(x_\epsilon)-v(x_\epsilon) \geq c(m)$, whose combination leads to $ u(x_\epsilon) \geq \alpha + \frac{c(m)}{2}  $ and $ v(x_\epsilon) < \alpha.$ Therefore, $f$ must satisfy $f(t) = a + b t $ for all $t \in (v(x_\epsilon), \beta),$ which contradicts the definition of $(\alpha, \beta)$. This proves that $ I \geq \alpha$. Now we show that $ S \leq \beta$. If $ \beta = + \infty$ we are done, so we suppose that $ \beta \in \R$. In this case we observe that $ \inf_\Omega u < \beta$ since otherwise we would have $u \geq \beta $ a.e. on $ \Omega$ and so the strong maximum principle (applied to $ u- \beta$)  would give $ u > \beta $ a.e. on $\Omega$ (recall that $v$ is not constant by assumption). Hence there would exist $\bar{x} \in \Omega $ such that $ v(\bar{x}) < \beta < u(\bar{x})$. This would imply that $f(t) = a + b t $ for all $t \in (\alpha, u(\bar{x})),$ contradicting the definition of $(\alpha, \beta)$. 
If $S > \beta$ (and since $\inf_\Omega u < \beta $ )  the same argument used to prove $\alpha \leq I$ can be applied to the solution $u$ to get a contradiction (again that $f$ would satisfy $f(t) = a + b t $ for all $t \in (\alpha, u(y)),$ for some $y \in \Omega$ such that $u(y) > \beta$). Therefore $ S \leq \beta$ and  $f(t) = a + b t $ for all $t \in (I, S).$ Hence, $u$ and $v$ are of class $C^2(\Omega)$ by standard elliptic estimates. Finally, $u-v \in H^1_0(\Omega)$ solves $ - \Delta (u-v) = f(u)-f(v) = b(u-v)$ in $ \Omega$ and $u-v>0$ in $\Omega$. Thererefore $u-v$ is a positive eigenfunction and so we necessarily have $ b = \lambda_1$. This concludes the proof. 
\hfill\qed
\end{proof}

\bigskip

\begin{theorem} \label{theorem:compa-stable-solutions-nonordered} Let $\Omega$ be a bounded domain of $\R^N$, $N \geq 1$ and let $f \in C^1([0;+\infty))$ be a convex function. Assume that $u,v \in H^1(\Omega)$ satisfy $ u-v \in H^1_0(\Omega),$ $f(u),f(v) \in L^1_{loc}(\Omega)$ and both $u$ and $v$ are solution to 
	\begin{equation}
	\left\{
	\begin{aligned}
	-\Delta w&= f(w) \qquad\text{in $\mathcal D'(\Omega)$}\\
	w&\geq 0 \qquad\text{a.e. on $\Omega$},
	\end{aligned}
	\right.
	\end{equation}
	If $u$ satisfies $f'(u) \in L^1_{loc}(\Omega)$ and $u$ is stable, then $u$ and $v$ are ordered, namely, one of the following three cases holds true : $u<v$ a.e. in $ \Omega$, $u \equiv v$ a.e. in $ \Omega$ or $ u > v$ a.e. in $ \Omega$. 
\end{theorem}

\begin{proof} By item ii) of Proposition \ref{prop:res-auxiliares} we have 	
\begin{equation}\label{buona-int}
(f(u) -f(v))(u-v)^+ \in L^1(\Omega)
\end{equation}
In view of the above results we can follow the proof of Theorem \ref{theorem:compa-stable-solutions} to obtain 
\begin{equation}\label{equazione-migliorata2}
\int_{\Omega} \vert \nabla (u-v)^+ \vert^2 =  \int_{\Omega} (f(u)-f(v))(u-v)^+.
\end{equation}
Using \eqref{stab-migliorata} with $ \varphi = (u-v)^+$, \eqref{equazione-migliorata2} and \eqref{disug-conv-migliorata2} we have 
\begin{equation*}
\int_{\Omega} f'(u)[(u-v)^+]^2 \leq  \int_{\Omega} \vert \nabla (u-v)^+ \vert^2 =  \int_{\Omega} (f(u)-f(v))(u-v)^+ \leq \int_{\Omega} f'(u)[(u-v)^+]^2 
\end{equation*}
and so
\begin{equation*}
\int_{\Omega} \vert \nabla (u-v)^+ \vert^2 - \int_{\Omega} f'(u)[(u-v)^+]^2 =0.
\end{equation*}
The latter and \eqref{stab-migliorata} imply that $(u-v)^+$ minimizes the the functional 
$\psi \longrightarrow \int_{\Omega} \vert \nabla \psi \vert^2 - \int_{\Omega} f'(u)\psi^2$ over $H^1_0(\Omega)$ and therefore $(u-v)^+$ solves 
\begin{equation}\label{ugual-stabilita}
- \Delta (u-v)^+ = f'(u)(u-v)^+ \qquad {\textit{in}} \quad \mathcal D'(\Omega).
\end{equation}
Then
\begin{equation}
- \Delta (u-v)^+ + [f'(u)]^-(u-v)^+ =  [f'(u)]^+(u-v)^+ \geq 0 \qquad {\textit{in}} \quad \mathcal D'(\Omega)
\end{equation}
with $[f'(u)]^- \in L^{\infty}(\Omega)$, since $ f'(u) \geq f'(0)$ by convexity of $f$. By the strong maximum principle, either $(u-v)^+ >0$ a.e. in $\Omega$ or $(u-v)^+ =0$ a.e. in $\Omega$. That is, either $u>v$ a.e. in $\Omega$ or $u \leq v$ a.e. in $\Omega$. In the latter case we have $ - \Delta (v-u) = f(v) -f(u) \geq f'(u)(v-u)$ in $\mathcal D'(\Omega)$ and so also $- \Delta (v-u) + [f'(u)]^-(v-u) \geq [f'(u)]^+(v-u) \geq 0$ in $\mathcal D'(\Omega),$ since $v-u \in H^1_0(\Omega)$ is nonnegative a.e. in $ \Omega$. As above, another application of the strong maximum principle  yields either $ u \equiv v$ or $ v> u $ a.e. on $\Omega$. 
\hfill\qed
\end{proof}

\bigskip

By combining Lemma \ref{lemma lambda1} and Theorem \ref{theorem:compa-stable-solutions} we immediately obtain the following classification result for stable solutions  in $H^1_0(\Omega)$.  

\begin{theorem}\label{th:stable:f(0)=0}
	Let $\Omega$ be a bounded domain of $\R^N$, $N \geq 1$ and let $f \in C^1(\R)$ be a 
	convex function such that $f(0)=0$. Assume that $u\in H^1_0(\Omega)$, $f(u) \in L^1_{loc}(\Omega)$ and that $u$ is a stable solution to
	\begin{equation}
	\left\{
	\begin{aligned}
	-\Delta u& = f(u) \qquad\text{in $\mathcal D'(\Omega)$}\\
	u&\geq 0 \qquad\text{a.e. on $\Omega$}.
	\end{aligned}
	\right.
	\end{equation} 
	Then, either $u \equiv 0$ or $f(t) = \lambda_1 t $ on $(0, \sup_\Omega u)$ and $u \in C^\infty(\Omega) \cap H^1_0(\Omega)$ is a positive first eigenfunction of $ -\Delta$ with homogeneous Dirichlet boundary conditions.   
\end{theorem}

\begin{remark}
If $u \equiv0$, then necessarily $ f'(0) \leq \lambda_1 $ by Lemma \ref{lemma lambda1}. Also observe that for any $\alpha \leq \lambda_1$ there is a convex function $f$ satisfying $f(0)=0$, $f'(0)=\alpha $ and such that $u \equiv 0$ is a stable solution to \eqref{eq:Lemma0}. An example is provided by $f(u) = u^2 +\alpha u$. 
\end{remark}

\begin{proof} 
	$v\equiv 0$ is a solution to \eqref{eq:Lemma0} since $f(0)=0$. Then, an application of Theorem \ref{theorem:compa-stable-solutions} provides the desired results. Indeed, since $u \in H^1_0(\Omega)$, we have $a=0$ (as observed in the footnote to Theorem \ref{theorem:compa-stable-solutions}). Therefore $f(t) = \lambda_1 t $ on $(\inf_\Omega v, \sup_\Omega u) = (0, \sup_\Omega u).$
	
	\hfill\qed
\end{proof}

\medskip

\begin{proposition}\label{prop: implica-stab}
Let $\Omega$ be a bounded domain of $\R^N$, $N \geq 1$ and let $f \in C^1([0;+\infty))$ be a convex function. Assume that $u,v \in H^1_{loc}(\Omega)$ satisfy $0 \leq v < u $ a.e. on $\Omega,$ $f(u),f(v) \in L^1_{loc}(\Omega)$ and 
\begin{equation}\label{sopra-sotto-implica-stab}
\left\{
\begin{aligned}
-\Delta u&\geq  f(u) \qquad\text{in $\mathcal D'(\Omega)$},\\
-\Delta v&\leq  f(v) \qquad\text{in $\mathcal D'(\Omega)$}.
\end{aligned}
\right.
\end{equation}
Then 
\begin{equation}\label{stab-migliorata2}
f'(v) \in L^1_{loc}(\Omega), \qquad  \int_{\Omega} f'(v) \varphi^2 \leq  \int_{\Omega} \vert \nabla \varphi \vert^2 \qquad \qquad \forall \varphi \in C^{\infty}_c(\Omega).
\end{equation}	 
In particular, if $v$ is a solution to $ -\Delta v = f(v)$ in ${\mathcal D'}(\Omega)$, then $v$ is stable. 
\end{proposition}

\begin{proof}
Recall that, by convexity of $f$, we have $ f(u) -f(v) \geq f'(v)(u-v)$ and $ [f'(v)]^- \leq [f'(0)]^-$ a.e. in $\Omega$. 
Therefore, $ [f'(v)]^- \in L^{\infty}(\Omega)$ and, using \eqref{sopra-sotto-implica-stab}, we obtain
\[
- \Delta (u-v) + [f'(v)]^-(u-v) \geq f(u)-f(v) + [f'(v)]^-(u-v) \geq [f'(v)]^+(u-v) \geq 0 \qquad in \quad \mathcal D'(\Omega).
\]
By the strong maximum principle and $ u-v>0$ a.e. in $\Omega$ we then get 
\begin{equation}\label{controllo-armonico-2}
\forall \, \omega \subset \subset \Omega \qquad u - v \geq  c(\omega)>0 \qquad  \text {a.e. on $\omega$,}
\end{equation}
where $c(\omega)$ is a positive constant depending on the open subset $\omega$. The latter implies that $ \frac{1}{u-v} \in L^{\infty}_{loc}(\Omega)$ and so $f'(v) \in L^1_{loc}(\Omega),$ thanks to item iii) of Proposition \ref{prop:res-auxiliares}. This proves the first claim of \eqref{stab-migliorata2}. To prove the second one we recall that 
$f'(v)(u-v) \in L^1_{loc}(\Omega)$ and use once again \eqref{sopra-sotto-implica-stab} to get
\begin{equation}\label{domina}
\int_{\Omega} \nabla (u-v) \nabla \phi \geq \int_{\Omega} f'(v)(u-v) \phi, \qquad \forall \, \phi \in C^{\infty}_c(\Omega).
\end{equation}
As before,  a standard approximation argument and Lebesgue's dominated convergence theorem yield that \eqref{domina} holds true for any $\phi \in H^1(\Omega) \cap L^{\infty}(\Omega)$ with compact support. 
Therefore, for every $ \varphi \in C^{\infty}_c(\Omega)$ and recalling \eqref{controllo-armonico-2}, we can then take $ \phi = \frac{\varphi^2}{u-v}$ in \eqref{domina} and find 
\begin{equation}\label{domina-2}
\int_{\Omega} \nabla (u-v) \nabla \Big(\frac{\varphi^2}{u-v} \Big) \geq \int_{\Omega} f'(v) \varphi^2, \qquad \forall \, \varphi \in C^{\infty}_c(\Omega).
\end{equation}
Hence,
\[
 \int_{\Omega} f'(v) \varphi^2 \leq \int_{\Omega}  2 \frac {\varphi \nabla (u-v) }{u-v} \nabla \varphi - 
 \int_{\Omega} \frac{\varphi^2}{(u-v)^2} \vert \nabla (u-v)\vert^2, \qquad \forall \, \varphi \in C^{\infty}_c(\Omega). 
\]
The second conclusion of \eqref{stab-migliorata2} then follows by applying Young's inequality to the first integral on the r-h-s of the latter inequality. The last claim is a consquence of \eqref{stab-migliorata2}. 
\hfill\qed
\end{proof}

\medskip

\noindent{\bf Proof of Theorem \ref{smoothness boundary}.}

(1) Since $u$ has finite Morse index, there exists a neighborhood of the boundary of the form $\Omega_\epsilon=\{x\in\Omega:\;\text{ dist\;}(x,\partial\Omega)<\epsilon\}$ such that $u$ is stable in $\Omega_\epsilon$ { (see point (5) below Definition \ref{loc stab}).} Thanks to Theorem \ref{Prop-approx-crucial}, it suffices to prove \eqref{boundary} in the case where $u\in C^{2,\alpha}(\overline\Omega_{\epsilon})$.
Also, the estimate will follow if we prove that for some $\rho\in(0,\epsilon)$ and for every $x\in\Omega_{\rho}$, there exists a set $I_{x}$ such that $\vert I_{x}\vert\ge\gamma$ and $u(x)\le u(y)$, for all $y\in I_{x}$.

To this end, we apply the moving-plane method. For $y\in\partial\Omega$, let $n(y)$ denote the unit normal vector to $\partial\Omega$, pointing outwards. Thanks to Lemmas 4.1 and 4.2 in \cite{ac}, there exists a constant $\lambda_0\in(0,\epsilon/2)$ depending on $\Omega$ only, such that 
$$
\{ x=y-tn(y)\;:\; 0<t<2\lambda_0, y\in\partial\Omega \} \subset \Omega
$$
In addition, in a fixed neighborhood of $\partial\Omega$, every point can be written in the form $x=y-tn(y)$, where $0<t<\lambda_0$ and $y$ is the unique projection of $x$ on $\partial\Omega$.
Fix $x_0\in\partial\Omega$ and $n=n(x_0)$. By applying the standard moving-plane method in the cap
$
 \Sigma_{\lambda}:=\left\{x\in\Omega \;:\;0<-(x-x_{0})\cdot n<\lambda \right\},
$
we deduce that 
\begin{equation}\label{mpm}
\partial_n u<0 \quad\text{in $\Sigma_\lambda$,} \quad \text{for every $\lambda\in[0,\lambda_0]$}.
\end{equation}
Next, since $\Omega$ is uniformly convex, there exists a radius $r>0$ depending on $\Omega$ only, such that the geodesic ball $B=B(n(x_0),r)\subset\mathbb S^{N-1}$ can be realized as the set of normals at nearby points and so $B\subset n(\partial\Omega)$.  
{
To see this, assume without loss of generality that $x_0=0$ and that $\partial\Omega$ coincides near $x_0$ with the graph of some $C^2$ function $\varphi:\R^{N-1}\to\R$ such that $\varphi(0)=0$, $\nabla\varphi(0)=0$ and $\nabla^2\varphi(0)$ is a diagonal matrix with eigenvalues bounded below by a positive constant (i.e. the directions of principal curvature of $\partial\Omega$ at $x_0$ coincide with the canonical basis of $\R^{N-1}$). Then, $n(x_0)=(0,\dots,0,1)$ and for $t=(t_1,0,\dots,0)\in\R^{N-1}$ small, there holds
$$n(\varphi(t))=\frac{(-\nabla\varphi(t),1)}{\sqrt{1+\vert\nabla\varphi(t)\vert^2}}=n(x_0)-t_1\left(\frac{\partial^2\varphi}{\partial t_1^2}(0), 0, \dots,0\right)+o(\vert t_1\vert)
$$
and so $n$ describes an arc of circle in the $x_1$ direction as $t_1$ varies in some small interval $(-r_1,r_1)$. This is also true (uniformly, since $\Omega$ is uniformly convex) in any direction $e\in \R^{N-1}$ and so a small geodesic ball $B=B(n(x_0),r)\subset\mathbb S^{N-1}$ can indeed be realized as the set of normals at nearby points.
}
This in turn implies that for all $\theta\in B$,
$$
\partial_{\theta}u<0\qquad\text{in $\Sigma:=\left\{x\in\Omega\;:\; \frac14\lambda_{0}<-(x-x_{0})\cdot n(x_{0})<\frac34\lambda_{0}\right\}$.}
$$
Indeed, applying the moving-plane procedure at every point $y\in\partial\Omega$ such that $\theta=n(y)$, $\theta\in B$, we have 
$$
\partial_{\theta}u<0\quad\text{in }\left\{x\in\Omega \;:\;0<-(x-y)\cdot \theta<{\lambda_{0}} \right\}.
$$
By taking a smaller ball $B$ if necessary, we may assume that 
$$
\vert (x-x_{0})\cdot(\theta - n(x_{0})) + (x_0-y)\cdot\theta\vert <\frac14\lambda_{0}, \quad\text{for all $x\in\Sigma$ and $\theta=n(y) \in B$.}
$$
Now,  since $-(x-y)\cdot\theta=-(x-x_0)\cdot n(x_0)-(x-x_0)\cdot(\theta-n(x_{0}))- (x_{0}-y)\cdot\theta$, we have for any $x\in\Sigma$,
$$
{\lambda_{0}} = \frac14\lambda_{0}+\frac34\lambda_{0}>-(x-y)\cdot\theta>\frac14\lambda_{0} -\frac14\lambda_{0}=0
$$
and so, as claimed, for any $x\in\Sigma$, there holds
$$
\partial_{\theta} u(x)<0.
$$
Now take $\rho=\lambda_{0}/8$. Fix a point $x\in\Omega_{\rho}=\{x\in\Omega\;:\; \text{\rm dist}(x,\partial\Omega)<\rho\}$ and let $x_{0}$ denote its projection on $\partial\Omega$. On the one hand, $u(x)\le u(x_{1})$, where $x_{1}=x_{0}-\rho n(x_{0})$. On the other hand, $u(x_{1})\le u(z)$, for all $z$ in the cone $I_{x}\subset\Sigma$ having vertex at $x_{1}$, opening angle $B$, and height $\lambda_{0}/2$ and the proof is complete.

\

{ (2) Fix $\epsilon>2\lambda_0>0$ as above. According to Theorem \ref{Prop-approx-crucial}, there exists a sequence of functions $u_k\in C^{2,\alpha}(\Omega_\epsilon)$ which are stable solutions of a semilinear elliptic equation in $\Omega_\epsilon$ and converge a.e. to $u$ in $\Omega_\epsilon$. If $N\le 9$, we may apply the interior estimate Theorem \ref{th:cfrs} to deduce that 
$$
\Vert u_k\Vert_{L^\infty(\Omega_{2\lambda_0}\setminus\overline{\Omega_{\lambda_0}})}\le M\Vert u_k\Vert_{L^1(\Omega_\epsilon)}\le M\Vert u\Vert_{L^1(\Omega)},
$$
for some constant $M$ depending on $\Omega$. Furthermore,  if $\Omega$ is convex, we know that $u_k$ is monotone in the normal direction i.e. \eqref{mpm} holds for $u=u_k$, and so the inequality remains true all of $\Omega_{2\lambda_0}$. Passing to the limit $k\to+\infty$ and using standard elliptic regularity, we deduce that $u\in C^{2,\alpha}(\overline{\Omega_{\lambda_0}})$ in this case. In the case where $f$ is nondecreasing on $\Omega$ is $C^3$, we can directly apply Theorem \ref{Prop-approx-crucial} combined to Theorem 1.5 in \cite{cfrs}. 
}

\

(3) We write a generic point in $\R^N$ as $(x,y)\in\R^{N-1}\times\R$. Let $B'$ be the unit ball in $\R^{N-1}$, $N-1\ge 10$ and $\Omega\subset\R^N$ the open set obtained by gluing the cylinder $B'\times(-1,1)$ to the unit half-ball centered at $(x,y)=(0,-1)$ and to the unit half-ball centered at $(x,y)=(0,1)$. Let 
$\lambda_n:[0,2]\to\R_+$ be a $C^2$ increasing concave function such that $\lambda_n(y)=ny$ for $y\in[0,1]$ and $\lambda_n(2)=	(n+\rho_n)$, where $\rho_n\to0$. Extend $\lambda_n$ as an odd function on $[-2,2]$.
Then, 
the domain $\Omega^n=\{(x,\lambda_n(y))\;:\; (x,y)\in\Omega\}$ is convex but clearly not uniformly, nor even strictly. We let $u_n$ be the minimal solution to \eqref{equazione} with nonlinearity $f(u)=2(N-3)e^u$ and domain $\Omega^n$. Since $\underline u=0$ and $\overline u=-2\ln(\vert x\vert)$ are ordered sub and supersolution to the problem, $u_n$ is well-defined, stable and
$$
0< u_n< -2\ln(\vert x\vert)\qquad\text{in $\Omega^n$.}
$$
It already readily follows that the average $\frac1{\vert\Omega^n\vert}\Vert u_n \Vert_{L^{1}{(\Omega^n)}}$ of $u_n$ remains bounded.
In addition, 
since $-2\ln(\vert x\vert)$ is a strict supersolution of the equation, $u_n$ cannot be an extremal solution and so $u_n$ is smooth and strictly stable, i.e. its linearized operator has positive first eigenvalue.

Recall that $\rho_n\to 0$
 and assume by contradiction that $\Vert u_n\Vert_{L^\infty((\Omega^n)_{\rho_n})}\le M$ for some constant $M>0$. For $(x,y)\in\Omega$, let $v_n(x,y)=u_n(x,\lambda_n(y))$. Then, $v_n=0$ on $\partial\Omega$ and, letting $\mu_n$ denote the inverse function of $\lambda_n$,
\begin{equation}\label{eqn}
-(\Delta_x + (\mu_n')^2\partial_y^2+\mu_n''\partial_y)v_n  = f(v_n) \qquad\text{in $\Omega$}.
\end{equation}
We claim that 
$$\Vert\nabla v_n\Vert_{L^\infty(\Omega)}\le K,$$ for some constant $K>0$. To see this, we begin by estimating $\vert\nabla v_n\vert$ on $\partial\Omega$. On the flat part of the boundary, we have a natural barrier: since $v_n<-2\ln(\vert x\vert)$ in $\Omega$ and $v_n=-2\ln(\vert x\vert)=0$ on $\partial B'\times(-1,1)$, we deduce that $\Vert\nabla v_n\Vert_{L^\infty(\partial B'\times(-1,1))}\le 2$. The function $\zeta(x,y)=1-\vert(x,y)-(0,1)\vert^2$ vanishes on the boundary of the half-ball centered at $(0,1)$ and satisfies 
$$
-(\Delta_x + (\mu_n')^2\partial_y^2+\mu_n''\partial_y)\zeta = 2(N-1)+2(\mu_n')^2-\mu_n''\partial_y\zeta \ge 2(N-1) \qquad\text{for $y\ge1$.}
$$
Hence, a constant multiple of $\zeta$ can be used as a barrier on the half-ball centered at $(0,1)$. Working similarly with the other half-ball, we deduce that $\Vert\nabla v_n\Vert_{L^\infty(\partial\Omega)}\le K$ on the whole boundary of $\Omega$. To extend the inequality to the whole of $\Omega$, we observe that any partial derivative $\partial_i v_n$ solves the linearized equation. Since $v_n$ is strictly stable (because this is the case for $u_n$), the linearized operator at $v_n$ has positive first eigenvalue. It follows that
$$
\Vert\nabla v_n\Vert_{L^\infty(\Omega)}\le\Vert\nabla v_n\Vert_{L^\infty(\partial\Omega)}\le K
$$
as claimed. Up to extraction, the sequence $(v_n)$ converges uniformly to some lipschitz-continuous function $v$ in $\overline\Omega$. In addition, $v=0$ on $\partial\Omega$, and for any $\varphi\in C^\infty_c(B'\times(-1,1))$,
$
\int_\Omega f(v_n)\varphi\;dx\to \int_\Omega f(v)\varphi\;dx
$
 while
$$
\int_\Omega v_n(-(\Delta_x + (\mu_n')^2\partial_y^2+\mu_n''\partial_y))\varphi\;dx = \int_\Omega v_n(-(\Delta_x + \frac1{n^2}\partial_y^2))\varphi\;dx \to \int_\Omega v(-\Delta_x \varphi) \;dx.
$$ 
In particular, the function $w(x)=v(x,0)\in H^1_0(B')$ is a weak stable solution to 
$$
-\Delta w= f(w)\qquad\text{in $B'$.}
$$
But so is $\overline u=-2\ln(\vert x\vert)$. By uniqueness of the extremal solution, we must have $w=\overline u$, which is impossible, since $w$ is bounded. Hence, up to extraction, $\Vert u_n\Vert_{L^\infty((\Omega^n)_{\rho_n})}\to+\infty$.

\appendix
\section{}

 \begin{proposition} \label{Prop-approx-crucial2} 
	Assume  that $\alpha \in (0,1)$ and $N \geq 2$. Let $\Omega$ be a bounded domain of $\R^N$ and let $f\in C^1([0,+\infty))$ be a nondecreasing and convex function. Assume that
	$u \in H^1(\Omega)$ 
	is a stable\footnote{Note that $f'(u)$ is a nonnegative Lebesgue measurable function (by our assumptions on $f$) and so the stability inequality \eqref{stab} has a meaning.} solution of \eqref{equazione} such that $u\ge 0$ a.e. in $\Omega$.
	\item 1) 
	There exists a sequence $(\varepsilon_n)$ of real numbers in $[0,1)$ such that $\varepsilon_n \searrow 0$ and a sequence $(u_n)$ of functions in $H^1(\Omega) \cap C^2(\Omega)$ such that $u_n$ is a stable weak solution\footnote{That is a function $u_n$ satisfying $\int_\Omega \nabla u_n \nabla \varphi = \int_\Omega (1-\epsilon_n) f(u_n) \varphi$, for all $ \varphi \in H^1_0(\Omega)$. } to 
	\begin{equation}\label{eq:un}
	\left\{
	\begin{aligned}
	-\Delta&u_n = (1- \varepsilon_n)f(u_n) \quad \textit{in} \quad \Omega,\\
	&u_n- u \in H^1_0(\Omega),\\
	&0 \leq u_n \leq u \qquad\text{a.e. on $\Omega$},
	\end{aligned}
	\right.
	\end{equation}
	and
	\begin{equation}\label{cv:un}
	u_n \longrightarrow u \quad \text{in} \quad H^1(\Omega), \qquad 
	\qquad u_n \longrightarrow u  \quad {a.e.} \,\, \text{on} \,\, \Omega.
	\end{equation}
	
	\medskip
	
	Also, if we assume in addition that $\Omega$ is of class $C^1,$ $T$ is a $C^{2,\alpha}$ open portion of $\partial \Omega$ and $u_{\vert_T} = 0$ (in the sense of the traces), then $u_n \in C^{2,\alpha}(\Omega')$ for any domain $ \Omega' \subset \subset \Omega \cup T$ (sufficiently small) and $u_n = 0$ in $T'$.

	\medskip
	
	\item 2)  	Assume in addition that $u \in H^1_0(\Omega)$.  Then, the sequence $(u_n)$ can be chosen in $H^1_0(\Omega) \cap C^2(\Omega)$. If we assume in addition that $\Omega$ is of class $C^1,$ $T$ is a $C^{2,\alpha}$ open portion of $\partial \Omega$ and $u_{\vert_T} = 0$ (in the sense of the traces), then $u_n \in C^2_0(\Omega')$ for any domain $ \Omega' \subset \subset \Omega \cup T$ (sufficiently small). In particular, if  $\Omega$ is of class $C^{2,\alpha}$, then $u_n \in C^2_0(\overline {\Omega}).$   
\end{proposition}
	
\begin{proof}	
	
i) We argue as in \cite {BCMR} and in subsection 3.2.2 of \cite{dup}. Nevertheless, our approach requires several non standard modifications due to the fact that we work in $\mathcal{D}'$ and that $u$ is merely in $H^1(B)$ (i.e., $u$ does not have "zero boundary value"). 
 
Given $\epsilon\in (0,1)$, define $\Phi_\eps:[0,+\infty)\to[0,+\infty)$ by 
\begin{equation*}
\int_0^{\Phi_\eps(t)} \frac{ds}{f(s)} = (1-\epsilon)\int_0^t \frac{ds}{f(s)}
\end{equation*}

Since $\Phi_\eps$ solves the initial value problem 
\begin{equation}
\begin{cases}
\Phi_\eps'(t)f(t) = (1-\eps)f(\Phi_\eps(t)) , \qquad t>0\\
\Phi\eps(0)=0,
\end{cases}
\end{equation}
we see that $\Phi_\eps\in C^2([0,+\infty))$ is increasing, concave and satisfies $ 0 < \Phi_\eps' (t)< 1 $ , $ 0 \leq \Phi_\eps(t) \leq t$ for all $t \geq 0$, $ \Phi_\eps'(0) = 1-\eps \in (0,1)$, $ \Phi_\eps''(0) = - \eps(1-\eps) \frac{f'(0)}{f(0)} \leq 0$ 
Also, using the concavity on $[0,+\infty)$ of the function $h(t) :=  \int_0^t \frac{ds}{f(s)}$ we get that 

\begin{equation}\label{controllo-concavo}
0 \leq f(\Phi_\eps(t)) \leq \frac{C(f)}{\eps} (1+t) \qquad \forall \, t \geq 0, 
\end{equation}
where $C=C(f)>0$ is a constant depending only on $f$.

Since $u \in H^1(\Omega)$ and \eqref{controllo-concavo} is in force, we have $U_\eps=\Phi_\eps(u)\in H^1(\Omega)$ and so $U_\eps$ is a weak supersolution 
\footnote {That is it satisfies  $\int_\Omega \nabla U_\eps \nabla \varphi \geq  \int_\Omega (1-\epsilon) f(U_\eps) \varphi$, for all $ \varphi \in H^1_0(\Omega)$, $ \varphi \geq 0$.  Indeed, in view of the above properties of $ \Phi_\eps$ we can extend it to a $C^2$ function on the entire real line $\R$ (still denoted by $\Phi_\eps$) such that $\Phi_\eps$ is nondecreasing and concave, $\Phi_\eps'$ is nonnegative and bounded on $\R$. Then we can apply a variant of Kato's inequality (see e.g. Lemma 3.2.1 in \cite{dup}) to get that $U_\eps$ is a supersolution in $\mathcal D'(\Omega)$. A standard density argument and Fatou's Lemma then give the desired conclusion.} to 
\begin{equation}\label{approx}
\left\{
\begin{aligned}
-\Delta &u_{\epsilon,1}=(1-\epsilon)f(u_{\epsilon,1})&\quad\text{in $\Omega$,}\\
&u_{\epsilon,1} - \Phi_\eps(u) \in H^1_0(\Omega),
\end{aligned}
\right.
\end{equation}
while $v=0$ is a weak subsolution to \eqref{approx}. In addition, we have $0 < U_\eps$ a.e. on $\Omega$, since $f>0$.  
Therefore by the (standard) method of sub and supersolution in $H^1$
we obtain a stable weak solution\footnote{This solution is obtained by using the standard method of monotone iterations in $H^1$ applied to the sequence $(v_k)_{k\geq1}$ defined by $- \Delta v_{k+1} = (1-\eps) f(v_k) $ in $\Omega$, $v_{k+1} \in \{ \, v \in H^1(\Omega)) \, : \, v- \Phi_\eps(u) \in H^1_0 \, \} := H^1_{ \Phi_\eps(u)}$ and starting with $v_1 =  \Phi_\eps(u)\in H^1_{ \Phi_\eps(u)}$, the supersolution. Note that the sequence is well-defined in $ H^1_{ \Phi_\eps(u)}$ and satisfies  $0 \leq v_{k+1} \leq v_k \leq  \Phi_\eps(u)$ a.e. on $\Omega$ thanks to $f' \geq 0$ and since $f( \Phi_\eps(u)) \in L^2(B)$ by \eqref{controllo-concavo}. Furthermore,  the stabilty of $u_{\eps,1}$ comes from the stabilty of $u$ and the fact that $f'$ is positive and nondecreasing. Indeed, $ \forall \varphi \in C^1_c(\Omega)$ we have $ \int_\Omega \vert \nabla \varphi \vert^2 \geq \int_\Omega f'(u) \varphi^2 \geq \int_\Omega (f'(U_\eps) \varphi^2 \geq \int_\Omega (f'(u_{\eps,1}) \varphi^2  \geq \int_\Omega ((1-\eps)f'(u_{\eps,1}) \varphi^2 $.} $u_{\epsilon,1} \in H^1(\Omega)$ of \eqref{approx} such that $0< u_{\epsilon,1} \leq U_\eps$ a.e. on $\Omega$. Furthermore, from \eqref{controllo-concavo} we get
\begin{equation} \label{improved-int}
0 \leq f(u_{\eps,1}) \leq  f(U_\eps) \leq  \frac{C(f)}{\eps} (1 +u) 
\end{equation}
and so $ f(u_{\eps,1}) \in L^2(\Omega)$. The latter implies $u_{\eps,1} \in H^2_{loc}(\Omega)$ by elliptic regularity, hence 
\begin{equation}\label{improve-reg}
u_{\eps,1} \in L^p_{loc}(\Omega) \qquad \forall \, p < \frac{2N}{N-4} \quad (p \leq \infty \quad {\text {if}} \,\, N \leq 3, \quad p < \infty \quad {\text {if}} \,\, N=4)
\end{equation}
by Sobolev imbedding. 

In what follows, for any integer $j \geq 0,$ we shell denote by $ \Phi_\eps^j $ the composition of $ \Phi_\eps$ with itself $j$ times ($\Phi_\eps^0 = Id$.)

Now we can repeat the same construction to  find a stable weak solution $u_{\eps,2} \in H^1(\Omega)$ to
\begin{equation}\label{approx2}
\left\{
\begin{aligned}
-\Delta &u_{\epsilon,2}=(1-\epsilon)^2 f(u_{\epsilon,2}) \quad\text{in $\Omega$,}\\
&u_{\epsilon,2} - \Phi_\eps^2(u) \in H^1_0(\Omega),
\end{aligned}
\right.
\end{equation}
such that $ 0< u_{\eps,2} \leq \Phi_\eps (u_{\eps,1}) \leq u_{\eps,1} \leq u$  a.e. on $\Omega$. Here we have used $ \Phi_\eps (u_{\eps,1}) \in H^1(\Omega)$ as supersolution and again $v=0$ as subsolution. Also note that $0 \leq \Phi_\eps^2(u) \leq \Phi_\eps (u_{\eps,1})$ on $ \partial \Omega$ in the sense of $H^1(\Omega)$, since $[\Phi_\eps^2(u) - \Phi_\eps (u_{\eps,1})]^+ \in H^1_0(\Omega).$ 
 In particular, by \eqref{controllo-concavo}, $0 \leq f(u_{\eps,2}) \leq  \frac{C(f)}{\eps} (1 +u_{\eps,1}) \in L^p_{loc}(\Omega)$ for any $p$ in the range \eqref{improve-reg}, and thus $ u_{\eps,2} \in L^q_{loc}(\Omega)$ for all $q < \frac{2N}{N-8}$ ($q \leq \infty$ if $ N \leq 7, \, q < \infty$ if $ N=8$). 
 Also note that $ 0< u_{\eps,2} \leq \Phi_\eps (u_{\eps,1}) \leq \Phi_\eps (\Phi_\eps(u)) = \Phi_\eps^2(u) \leq  \Phi_\eps(u)$ a.e. on $\Omega$.
 
By iteration, we find that if $ k= \Big[\frac{N}{4}\Big]+1$, the  solution $u_{\eps,k} \in H^1(\Omega)$ to
\begin{equation}\label{approxk}
\left\{
\begin{aligned}
-\Delta&u_{\epsilon,k}=(1-\epsilon)^k f(u_{\epsilon,k})\quad\text{in $\Omega$,}\\
&u_{\epsilon,k} - \Phi_\eps^k(u) \in H^1_0(\Omega),
\end{aligned}
\right.
\end{equation}
is locally bounded (hence of class $C^2$ inside $\Omega$) and also satisfies $ 0< u_{\eps,k} \leq \Phi_\eps (u_{\eps,k-1})\leq \Phi_\eps^k(u) \leq u$ a.e. on $\Omega$.

Since $ \eps \in (0,1)$ is arbitrary we have proved that, for every $ \delta \in (0,1)$ (choose $\delta = 1 - (1-\eps)^k$) there exists a nonnegative stable weak solution $u_\delta \in H^1(\Omega) \cap C^2(\Omega)$ to
\begin{equation}\label{approx-delta}
\left\{
\begin{aligned}
-\Delta&u_\delta=(1-\delta)f(u_\delta)\quad\text{in $\Omega$,}\\
&u_{\delta} - \Phi_\delta^k(u) \in H^1_0(\Omega).
\end{aligned}
\right.
\end{equation}

Since $ 0 \leq u_\delta \leq  \Phi_\delta^k(u) \leq u$ a.e. on $\Omega$ by construction, we get $ \Vert u_\delta \Vert_{L^2(\Omega)} \leq \Vert \Phi_\delta^k(u)\Vert_{L^2(\Omega)} \leq  \Vert u \Vert_{L^2(\Omega)}$ and also that $ \Phi_\delta^k(u) \longrightarrow u$ in $L^2(\Omega)$ by the dominated convergence theorem (recall that $\Phi_\delta(t) \longrightarrow t$ for all $t \geq 0$ and that $\Phi_\eps$ is a contraction on $\R^+$).  Moreover, by choosing $ \Phi_\delta^k(u) - u_\delta \in H^1_0(\Omega)$ as test function in the weak formulation of \eqref{approx-delta} we obtain 

\[
\int_\Omega \nabla u_\delta \nabla ( \Phi_\delta^k(u) - u_\delta) = \int_\Omega (1-\delta) f(u_\delta) ( \Phi_\delta^k(u) - u_\delta)  \geq 0
\]
since $  \Phi_\delta^k(u) - u_\delta \geq 0$ a.e. on $\Omega$ and $ f \geq 0$. Therefore we deduce that $ \int_\Omega \vert \nabla u_\delta \vert^2 \leq \int_\Omega \nabla u_\delta \nabla  \Phi_\delta^k(u) $ which leads to $ \Vert \nabla u_\delta \Vert_{L^2(B)} \leq \Vert \nabla  \Phi_\delta^k(u) \Vert_{L^2(B)} $ by Young's inequality. On the other hand $ \nabla \Phi_\delta^k(u) = \Big (\prod_{j=0}^{k-1} \Phi_\delta' (\Phi_\delta^{(j)}(u)) \Big) \nabla u$, which entails 
$ \Vert \nabla u_\delta \Vert_{L^2(\Omega)} \leq \Vert \nabla \Phi_\delta^k(u) \Vert_{L^2(\Omega)} \leq \Vert \nabla u \Vert_{L^2(\Omega)}$. Therefore 
\begin{equation}\label{stima-energia}
\Vert u_\delta \Vert_{H^1(\Omega)} \leq \Vert \Phi_\delta^k(u) \Vert_{H^1(\Omega)} \leq  \Vert u \Vert_{H^1(\Omega)}.
\end{equation}
In particular the families $(u_\delta) $ and $( \Phi_\delta^k(u))$ are bounded in $ H^1(\Omega)$ and therefore, we may and do suppose that (up to subsequences) $u_\delta \rightharpoonup v$  in $H^1(\Omega)$, $u_\delta \longrightarrow v $ in $L^2(\Omega)$, $ u_\delta \longrightarrow v$ a.e. on $\Omega$ and $ \Phi_\delta^k(u)
\rightharpoonup V$  in $H^1(\Omega)$, $ \Phi_\delta^k(u) \longrightarrow V $ in $L^2(\Omega)$, $  \Phi_\delta^k(u) \longrightarrow V$ a.e. on $\Omega$, for some $v, V \in H^1(\Omega)$, as $\delta \to 0.$  From those properties we get $V=u$ (recall that $ \Phi_\delta^k(u) \longrightarrow u$ in $L^2(\Omega)$) and also that $v$ is a solution of $ - \Delta v = f(v) $ in $\Omega$. Also $v$ is stable thanks to $ u_\delta \longrightarrow v$ a.e. on $\Omega$, the positivity and the continuity of $f'$ and Fatou's Lemma. 
On the other hand, the weak convergence of $(u_\delta)$ and $(\Phi_\delta^k(u))$ in $H^1(\Omega)$ and \eqref{approx-delta} imply ${u_\delta} - \Phi_\delta^k(u) \rightharpoonup  v-u$ in $ H^1_0(\Omega)$. 
Finally  we get that $u_\delta \longrightarrow u$ in $H^1(\Omega)$, since $u_\delta \rightharpoonup u $ in $H^1(\Omega)$ and $\limsup  \Vert u_\delta \Vert_{H^1(\Omega)} \leq \Vert u \Vert_{H^1(\Omega)}$ by \eqref{stima-energia}. 
Since both $u$ and $v$ are stable solutions in $H^1(\Omega)$ and $ 0 \leq v \leq u$ a.e. on $\Omega$, we deduce from Theorem \ref{theorem:compa-stable-solutions} that either $u=v$ in $\Omega$ or $ u>0$ and $u \in C^\infty(\Omega)$. In the first case the desired conclusion follows by taking $ \epsilon_n = \delta_n$, where $ (\delta_n)$ is any sequence in $(0,1)$ such that $ \delta_n \searrow 0$ and $ u_n = u_{\delta_n},$ while in the second one it is enough to take $ \eps_n = 0$, $ u_n = u$ for every $ n \geq 1$.  

The last claim of item i) then follows by standard elliptic theory. 

\medskip

ii) If $f(0)>0,$ the conclusion follows from item i).  If $f(0)=0$, either $u\equiv 0$ or $u$ is a positive first eigenfunction of $ -\Delta$ with homogeneous Dirichlet boundary conditions, by Theorem \ref{th:stable:f(0)=0}. 
In both cases the smoothness of $u$ up to (a portion of the) boundary follows from the elliptic regularity, since $\partial \Omega$ is smooth enough. To conclude it is enough to take $ \eps_n = 0$, $ u_n = u$ for every $ n \geq 1$. 

\hfill\qed 

\end{proof}

\end{document}